\documentclass[12pt,reqno]{amsart}
\usepackage{a4wide}

\numberwithin{equation}{section}
\usepackage{mathrsfs}
\usepackage{amsfonts}
\usepackage{amsmath}
\usepackage{stmaryrd}
\usepackage{amssymb}
\usepackage{amsthm}
\usepackage{mathrsfs}
\usepackage{url}
\usepackage{amsfonts}
\usepackage{amscd}
\usepackage{indentfirst}
\usepackage{enumerate}
\usepackage{amsmath,amsfonts,amssymb,amsthm}
\usepackage{amsmath,amssymb,amsthm,amscd}
\usepackage{graphicx,mathrsfs}
\usepackage{appendix}

\usepackage[numbers,sort&compress]{natbib}

\usepackage{color}
 \usepackage[colorlinks, linkcolor=blue, citecolor=blue]{hyperref}

\newcommand{\R}{\mathbb{R}}

\newtheorem{Thm}{Theorem}[section]
\newtheorem{Lem}[Thm]{Lemma}
\newtheorem{Cor}[Thm]{Corollary}
\newtheorem{Prop}[Thm]{Proposition}

\theoremstyle{definition}
\newtheorem{Def}[Thm]{Definition}
\theoremstyle{remark}

\begin{document}

\title
[Energy critical half-wave]%每页页眉小标题
{Multiplicity and asymptotics of standing waves for the energy critical half-wave}%文章完整的标题
%11901147 罗肖师兄基金
%12071169 王春花老师基金（杨涛参与者）
%2020CXZZ070 杨涛创新
\footnotetext{*This work was supported by National Natural Science Foundation of China (Grant Nos. 11901147, 12071169), the Fundamental Research Funds for the Central Universities of China (Grant No. JZ2020HGTB0030) and the postgraduate education innovation funding project from Central China Normal University (Grant No. 2020CXZZ070).}

%\author[D. Cao, P. Luo and S. Peng]{Daomin Cao, Peng Luo and Shuangjie Peng}

\maketitle
\begin{center}
\author{Xiao Luo}
\footnote{Email addresses: luoxiao@hfut.edu.cn (X. Luo).}
%\author{Tao Yang\textsuperscript{*}}
%\let\thefootnote\relax\footnotetext{* Corresponding Author: yangtao\_pde@163.com (T. Yang).}
\author{Tao Yang}
\footnote{Email addresses: yangtao\_pde@163.com (T. Yang).}

\author{Xiaolong Yang}
\footnote{Corresponding Author: yangxiaolong@mails.ccnu.edu.cn (X. L. Yang).}

\end{center}

\begin{center}
\address {1 School of Mathematics, Hefei University of Technology, Hefei, 230009, P. R. China}

\address {2,3 School of Mathematics and Statistics, Central China Normal University, Wuhan, 430079, P. R. China}
\end{center}
%\author{Xiaolong Yang\textsuperscript{*}}
%\address[Xiaolong Yang]{School of Mathematics and Statistics, Central China Normal University,
%Wuhan 430079, China}
%\email{yangxiaolong@mails.ccnu.edu.cn}

%\address[Daomin Cao]{Institute of Applied Mathematics, AMSS, The Chinese Academy of Sciences, Beijing 100190, China.
%School of Mathematics and Information Science
%Guangzhou University
%Guangzhou 510405, China}
%\email{dmcao@amt.ac.cn}

%\address[Peng Luo]{School of Mathematics and Statistics and Hubei Key Laboratory of Mathematical Sciences, Central China Normal University, Wuhan 430079, China}
%\email{pluo@mail.ccnu.edu.cn}

%\address[Shuangjie Peng]{School of Mathematics and Statistics and Hubei Key Laboratory of Mathematical Sciences, Central China Normal University, Wuhan 430079, China}
%\email{sjpeng@m/Users/haixiachen/Library/Containers/com.tencent.xinWeChat/Data/Library/Application Support/com.tencent.xinWeChat/2.0b4.0.9/ddd7b7e5a51f2a5fc420ee22dbc85f30/Message/MessageTemp/bceb85552710ac0c22847c215a76003e/File/CLTX-2021(2).texail.ccnu.edu.cn}

%}
\begin{abstract}
In this paper, we consider the multiplicity and asymptotics of standing waves with prescribed mass $\int_{{\mathbb{R}^N}} {{u}^2}=a^2$ to the energy critical half-wave% perturbed by a focusing mass-subcritical term. This leads to the study of the following nonlocal elliptic equation
\begin{equation}\label{eqA0.1}
\sqrt{-\Delta}u=\lambda u+\mu|u|^{q-2} u+|u|^{2^*-2}u,\ \ u\in H^{1/2}(\R^N),
\end{equation}
where $N\!\geq\! 2$, $a\!>\!0$, $q \!\in\!\big(2,2+\frac{2}{N}\big)$, $2^*\!=\!\frac{2N}{N-1}$ and $\lambda\!\in\!\R$ appears as a Lagrange multiplier. %This leads to the study of the constraint where the frequency $\lambda\in\mathbb{R}$ is unknown and appears as Lagrange multiplier.
%We show that the above problem admits a local minima of the associated Energy functional. By using a refined version of the min-max principle, there also exist second solution which are located a mountain-pass level of the Energy functional.
We show that \eqref{eqA0.1} admits a ground state $u_a$ and an excited state $v_a$, which are characterised by a local minimizer and a mountain-pass critical point of the corresponding energy functional.
% on the mass-sphere,respectively.
Several asymptotic properties of $\{u_a\}$, $\{v_a\}$ are obtained and it is worth pointing out that we get a precise description of $\{u_a\}$ as $a\!\to\! 0^+$ without needing any uniqueness condition on the related limit problem. The main contribution of this paper is to extend the main results in J. Bellazzini et al. [Math. Ann. 371 (2018), 707-740] from energy subcritical to energy critical case. Furthermore, these results can be extended to the general fractional nonlinear Schr\"{o}dinger equation with Sobolev critical exponent, which generalize the work of H. J. Luo-Z. T. Zhang [Calc. Var. Partial Differ. Equ. 59 (2020)] from energy subcritical to energy critical case.
\end{abstract}

%\date{\today}
\maketitle

{\small\small
\section{Introduction and Main Result}

\setcounter{equation}{0}

The nonlinear half-wave equations
%with combined nonlinearities
\begin{equation}\label{eq1}
i\partial_{t} \psi=\sqrt{-\Delta} \psi-f(|\psi|)\psi  \text { in } \mathbb{R}\times \mathbb{R}^{N}
\end{equation}
are well-known and widely studied by scholars, see \cite{GL,BJ,BGL,HW,KL} and references therein. The operator $\sqrt{-\Delta}$ is defined by%on $H^{1/2}(\R^N)$ by
%via the Fourier transform, i.e.
\begin{equation*}
\sqrt{-\Delta}\psi = \mathcal{F}^{-1}(|\xi|\mathcal{F}\psi),\ \ \forall \psi\in H^{1/2}(\R^N),
\end{equation*}
where $\mathcal{F}\psi$ is the Fourier transform of $\psi$, the space $H^{1/2}(\R^N)$ is defined as
$$
H^{1/2}(\R^N):=\big\{u\in L^2(\R^N),\ \int_{\R^N}\int_{\R^N}\frac{|u(x)-u(y)|^2}{|x-y|^{N+1}}<+\infty\big\}.
$$
Physically, the motivation to the study of \eqref{eq1} ranges from turbulence phenomena, wave propagation, continuum limits of long-range lattice system, and models for gravitational collapse in astrophysics (see e. g. \cite{EA,ElMl}).
%The problem has been used to model the dynamics of pseudo-relativistic boson stars, see \cite{ElMl}.
%This equation has rich mathematical problems and has attracted a lot of attention recently(see \cite{GL,BJ,BGL,DV,DVF,HW,KL} and references therein).
%We mention some results on power-type nonlinearities, see \cite{BJ} for traveling solitary waves.%in the Sobolev space $H^{1/2}(\R^N)$
%have been investigated in  e.g., \cite{BGL,HW}.
Equation \eqref{eq1} with power-type nonlinearities has been studied in \cite{BJ}. Concerning the local/global well-posedness of solutions to \eqref{eq1}, please refer to \cite{BGL,HW}. The dynamical properties of blow-up solutions have been investigated in  e.g. \cite{FB,KL}.

Taking $f(t)=\mu t^{q-2}+t^{p-2}$ with $2<q,p\leq\frac{2N }{N-1}$, if $u$ satisfies the following equation
\begin{equation}\label{1.1}
\sqrt{-\Delta}u=\lambda u+\mu|u|^{q-2} u+|u|^{p-2}u, \ \ \ u\in H^{1/2}(\R^N)
\end{equation}
for some $\lambda\in \R$, then
$$
\psi(t,x)=e^{-i\lambda t}u(x)
$$
solves equation \eqref{eq1}. We say that $u\in H^{1/2}(\mathbb{R}^{N})$ is a weak solution to (\ref{1.1}) if
$$\int_{\mathbb{R}^{N}} [\sqrt{\!-\!\Delta} u \varphi\!-\!\lambda u{\varphi}]\!-\!\mu\int_{\mathbb{R}^{N}} \left|u\right|^{q\!-\!2} u {\varphi}\!-\! \int_{{\mathbb{R}^N}}|u|^{p-2}u\varphi\!=0,~~~~\forall \varphi \in H^{1/2}(\mathbb{R}^{N})$$
and $u\in H^{1/2}({\mathbb{R}^N})$ is a \textbf{normalized solution} to (\ref{1.1}) if there exist a real number $\lambda \in {\mathbb{R}}$ and a function $u\in  H^{1/2}({\mathbb{R}^N})$ weakly solving (\ref{1.1}) together with the normalization condition $||u||_2^2=a^2$.

Recently, normalized solutions to elliptic PDEs and systems attract much attention of researchers e.g. \cite{JJ,JL,LY,NsAe,NSaE,LZ}. In \cite{NSaE}, N. Soave considered the existence of normalized ground states to the following energy (Sobolev) critical Schr\"{o}dinger equation
\begin{equation}\label{eq2}
-\Delta u=\lambda u+\mu |u|^{q-2}u+|u|^{\frac{4}{N-2}}u, \ \  \ u\in H^{1}(\R^N),
\end{equation}
where $\mu\in\R$ and $2<q<\frac{2N}{N-2}$. The whole study of \cite{NSaE} can be considered as a counterpart of the Br\'{e}zis-Nirenberg problem in the context of normalized solutions. If $\mu>0$ and $2<q<2+\frac{4}{N}$, it was proved in \cite{NSaE} and \cite{JJ} that for small mass, there is a ground state to \eqref{eq2}, which corresponds to the local minimizer of the associated energy functional constrained on the $L^2$-sphere via different methods, respectively. In this case, N. Soave \cite{NSaE} raised an open question on how to obtain the second solution of \eqref{eq2} since the associated energy functional constrained on the $L^2$-sphere admits a convex-concave geometry. It is worth pointing out that, L. Jeanjean and T. T. Le \cite{JL} solved this open question if $N \ge 4$. Also for small mass, Jeanjean et al. in \cite{JL} proved the existence of a second solution to \eqref{eq2}, which is an excited state corresponding to the mountain-pass critical point of the associated energy functional.

H. J. Luo and Z. T. Zhang \cite{ZL} extended the existence and nonexistence results of N. Soave \cite{NsAe} to fractional Schr\"{o}dinger equation
\begin{equation}\label{ab1}
(-\Delta)^\sigma u=\lambda u+\mu |u|^{q-2}u+|u|^{p-2}u, \ \ u\in H^{\sigma}(\R^N),
\end{equation}
%under the constraint $\int_{\R^N}|u|^2=a^2$,
where $0<\sigma<1$, $N\ge2$, $\mu\in\R$ and $2<q<p<\frac{2N}{N-2\sigma}$.
%under different assumptions on *******%分数次指标 $2<q<p<\frac{2N}{N-2s}$?
In addition, they proved a uniqueness result for the case $\mu=0$ and
$2<p<2+\frac{4\sigma}{N}$.
%studied the existence and nonexistence of normalized solutions for the critical fractional equation \eqref{ab1} with combined nonlinearities.
If $\sigma=\frac{1}{2}$, $N\ge 2$, $\mu=0$ and $2+\frac{2}{N}<p<\frac{2N}{N-1}$, Bellazzini et al. in \cite{BJ} considered a minimization problem for \eqref{ab1} under the normalization constraint on the Pohozaev manifold. Moreover, they presented a different proof of the existence of ground states and characterized the dynamics near the ground states. In preparing this paper, we notice that in the very recent work \cite{ZZ}, the authors extended the results of \cite{NSaE} to the fractional Laplacian operator and studied \eqref{ab1} with $p=\frac{2N}{N-2\sigma}$, they also obtained a ground state to \eqref{ab1} for small mass.

%The fractional Schr\"{o}dinger equation
%\begin{equation}\label{eq1.0}
%i\partial_{t} \psi+(-\Delta)^{\sigma} \psi=f(|\psi|)\psi \text { in } \mathbb{R}\times \mathbb{R}^{N},
%\end{equation}
%is a fundamental equation of the space-fractional quantum mechanics, see \cite{Nlan}. Here the fractional Laplacian $(-\Delta)^{\sigma}$ is defined as
%$$
%(-\Delta)^{\sigma} u(x)=C_{N,\sigma} \mathrm{P.V.} \int_{\mathbb{R}^{N}} \frac{u(x)-u(y)}{|x-y|^{N+2\sigma}} d y, \forall u \in S(\mathbb{R}^{N}),
%$$
%where $S(\mathbb{R}^{N})$ denotes the Schwartz space of rapidly decreasing smooth functions, P.V. stands for the principle value of the integral and $C_{N,\sigma}$ is some positive normalization constant.
%For $\sigma=1/2$, (\ref{eq1.0}) has been also used as models to describe Boson-stars, see \cite{ElMl}.
%In \cite{Slhi}, S. Longhi proposed an optical realization of the fractional Schr\"odinger equation. Refer to \cite{SpeR,RpSe} for more information about  physical backgrounds on (\ref{eq1.0}).
Motivated by the above works, in this paper, we consider the existence and asymptotic properties of solutions to the energy critical half-wave
\begin{equation}\label{eq1.1}
\sqrt{-\Delta} u=\lambda u+\mu|u|^{q-2} u+ |u|^{ 2^*-2 }u, \ \ u\in H^{1/2}(\mathbb{R}^{N})
\end{equation}
under the constraint
\begin{equation}\label{ad1}
 \int_{\R^N}u^2=a^2,
\end{equation}
%\begin{equation}\label{eq1.2}
%\int_{\R^N}|u|^2=a^2,
%\end{equation}
where
$\mu>0,~~q \!\in\!\big(2,2+\frac{2}{N}\big),~~2^*:=\frac{2N }{N-1}$.
Notice that $|u|^{2^*-2}u$ is energy critical as $2^*$ is the Sobolev critical exponent. A natural approach to find solutions of \eqref{eq1.1}-\eqref{ad1} is to search critical points of the energy functional
%Normalized solutions to (\ref{eq1.1}) can be obtained by searching critical points of the energy functional
\begin{align*}
F_{\mu}(u)=\frac{1}{2}\int_{\R^N}u\sqrt{-\Delta} u-\frac{\mu}{q}\int_{{\mathbb{R}^N}}|u|^{q}-\frac{1}{2^*}
\int_{{\mathbb{R}^N}}|u|^{2^*}
\end{align*}
on the $L^2$-sphere
\begin{equation*}\label{ab2}
S(a): =\Big\{ u \in H^{1/2}({\mathbb{R}^N}): {||u||}_2^2=a^2  \Big\}.
\end{equation*}
 Then the parameter $\lambda$ appear as Lagrange multipliers.

%with $\lambda$ Lagrange multipliers.

By the $L^2$-norm preserving dilations $t\star u(x)=t^{\frac{N}{2}}u(tx)$ with $t\!>\!0$, it is easy to know that
$\bar{q}:=2+\frac{2}{N}$ is the $L^2$-critical exponent of \eqref{eq1.1}.
%Indeed, for any $\mu>0$, we have
%$$\inf _{u \in S_{a}} F_{\mu}(u)\!=\!-\infty,~~~~\mbox{if}~~~~\bar{q}<q\leq\frac{N+1}{N-1} $$
%$$\inf _{u \in S_{a}} F_{\mu}(u)\!>\!-\infty,~~~~\mbox{if}~~~~1<q<\bar{q}.$$
%In this paper, we consider the existence and asymptotic properties of normalized solutions to (\ref{eq1.1}) with
%Problem (\ref{eq1.1})-(\ref{eq1.2}) arises from seeking standing waves
%for the following nonlinear fractional Schr\"{o}dinger equation
%\begin{equation}\label{eq1.3}
%i\partial_{t} \psi+\sqrt{-\Delta} \psi=\mu|\psi|^{q-1} \psi+|\psi|^{\frac{2}{N-1}}\psi \  \text { in } \ \mathbb{R}\times \mathbb{R}^{N}.
%\end{equation}
%A standing wave of (\ref{eq1.3}) is a solution having the form ${\psi}(t,x)=e^{-i {\lambda} t}u(x)$ for some ${\lambda}\in \mathbb{R}$ and $u$ satisfying (\ref{eq1.1}). So (\ref{eq1.1}) is the stationary equation of the time-dependent equation (\ref{eq1.3}).
%, which is a special case of (\ref{eq1.0}).
%That is, we are in the setting $\sigma\in(0,1)$ and $\inf _{u \in S_{a}} F_{\mu}(u)\!=\!-\infty$.
%To state our main results, we give some notations and frequently used constants.
We say that $\tilde{u}$ is a ground state of (\ref{eq1.1}) on $S(a)$ if
$$
d\left.F_{\mu}\right|_{S(a)}(\tilde{u})=0 \quad \text { and } \quad F_{\mu}(\tilde{u})=\inf \big\{F_{\mu}(u): ~~~~d\left.F_{\mu}\right|_{S(a)}(u)=0,\ \text{and}\ u\in S(a)\big\}.
$$
As $2<q<\bar{q}<2^*$, we have $\inf\limits_{u\in S(a)}F_\mu(u)=-\infty$. In spirit of Jeanjean \cite{JJ}, we see that the presence of the mass subcritical term $\mu|u|^{q-2} u$ created a convex-concave geometry of $F_\mu|_{S(a)}$ if $a>0$ is small. In view of this, it is natural to introduce a suitable local minimization problem
\begin{equation}\label{b11}
m(a):=\inf_{u\in V(a)}F_\mu
(u)<0\le\inf_{u\in\partial V(a)}F_\mu(u),
\end{equation}
where
$$
V(a):=\{u\in S(a): \|u\|_{\dot{H}^{1/2}}<\rho_0 \},~~\partial V(a):=\{u\in S(a): \|u\|_{\dot{H}^{1/2}}=\rho_0 \},
$$
see Section 3 for details. Up to a normalizing constant, we write $\|u\|^2_{\dot{H}^{1/2}}$ as
\begin{equation*}
\|u\|^2_{\dot{H}^{1/2}}=\int_{\R^N}u\sqrt{-\Delta} u=\int_{\R^N}\int_{\R^N}\frac{|u(x)-u(y)|^2}{|x-y|^{N+1}}.
\end{equation*}

Our main results are as follows. We first consider the existence and asymptotic behaviour of ground states to \eqref{eq1.1}-\eqref{ad1}. Let $Q$ be the unique positive radial ground state of $ \sqrt{-\Delta} Q+Q-|Q|^{q-2}Q=0$, we have
\begin{Thm}\label{th1.1}
Let $N \!\ge \! 2$, $\mu>0$ and $q \in\big(2,2+\frac{2}{N}\big)$, then there exists a $a_*\!=\!a_*(\mu)\!>\!0$ such that, for any $a\in (0,a_*]$, $m(a)$ is achieved by some $u_a\in V(a)$ and $u_a$ is a ground state of \eqref{eq1.1}-\eqref{ad1}. Furthermore, we have
\begin{enumerate}
\item   For any fixed $a >0$, $u_a \in S(a)$ exists for any $\mu >0$ sufficiently small and
\begin{equation*} % \quad \mbox{and} \quad
\|u_a\|_{\dot{H}^{1/2}}^2 \to 0,\ \ F_{\mu}(u_a) \to 0 \quad \text{as} \ \mu \to 0;
\end{equation*}
\item  For any fixed $\mu >0$, letting $a\to 0$, then for any ground state $u_{a}\in V(a)$, we have
\begin{equation*}
\frac{m(a)}{a^{\frac{2q(1-\gamma_q)}{2-q\gamma_q}}}\to -K_{N,q},\ \
\frac{\lambda_{a}}{a^{\frac{2(q-2)}{2-q\gamma_q}}}\to
-\frac{2q(1-\gamma_q)}{2-q\gamma_q}K_{N,q},\ \
\frac{\|u_{a}\|^2_{\dot{H}^{1/2}}}{a^{\frac{2q(1-\gamma_q)}{2-q\gamma_q}}}\to \frac{2q\gamma_q}{2-q\gamma_q}K_{N,q}.
\end{equation*}
In addition, we have
\begin{equation*}
\frac{1}{\alpha} u_{a}\big(\frac{x}{\beta}\big)\to Q \ \  \text{in}\ H^{1/2}(\R^N )\ \ \text{as}\ a\to 0^+,
\end{equation*}
where $\alpha\!=\!\frac{a\beta^{\frac{N}{2}}}{\|Q\|_2}$, $\beta\!=\!\Big[ \frac{\mu a^{q-2}}{\|Q\|^{q-2}_2} \Big]^{\frac{2}{2-q\gamma_q}}$, $K_{N,q}\!=\!\frac{(2-q\gamma_q)\gamma_q}{2q(1-\gamma_q)}\cdot \|Q\|^{\frac{2(2-q)}{2-q\gamma_q}}_2\mu^{\frac{2}{2-q\gamma_q}}\!>\!0$ and $\gamma_q\!=\!\frac{N(q-2)}{q}$.
\end{enumerate}
%$F_\mu$ restricted to $S(a)$ has a ground state. This ground state is a local minimizer of $F_\mu$ in the set $V(a)$.
\end{Thm}

Next, we study the existence and asymptotic behaviour of a second solution of \eqref{eq1.1}-\eqref{ad1}. Denote
$$\mathcal{S}=\inf _{u \in {\dot{H}}^{1/2}(\mathbb{R}^{N})\setminus \{0\} }    \frac{\left\| u\right\|_{\dot{H}^{1/2}}^{2}}{||u||_{2^{*}}^{2}}, $$
refer to Section 2 for details (see also \cite{EDgE}).  We have

\begin{Thm}\label{th1.2}
Let $N \!\ge \! 2$, $q \in\big(2,2+\frac{2}{N}\big)$, $\mu>0$ and $a_*=a_*(\mu)>0$ given by Theorem \ref{th1.1}. Then, for any $a\in (0,a_*]$, there exists a second solution $v_a\in S(a)$ to \eqref{eq1.1}-\eqref{ad1} such that
\begin{equation*}
0<F_\mu(v_a)<m(a)+\frac{\mathcal{S}^N}{2N}.
\end{equation*}
Furthermore, we have
\begin{enumerate}
\item  For any fixed $\mu >0$, it holds that
\begin{equation*}
\|v_a\|_{\dot{H}^{1/2}}^2 \to \mathcal{S}^{N},~~~~ F_{\mu}(v_a) \to \frac{\mathcal{S}^{N}}{2N} \quad \text{as} \quad a \to 0;
\end{equation*}
\item   For any fixed $a >0$, $v_a \in S(a)$ exists for any $\mu >0$ sufficiently small and
\begin{equation*}
\|v_a\|_{\dot{H}^{1/2}}^2 \to \mathcal{S}^{N},~~~~ F_{\mu}(v_a) \to \dfrac{\mathcal{S}^N}{2N} \quad \text{as}  \quad \mu \to 0.
\end{equation*}
\end{enumerate}
\end{Thm}

%\begin{Thm}\label{th1.3}
%Under the assumptions of Theorem \ref{th1.2}, we have
%\begin{enumerate}
%\item  For any fixed $\mu >0$, it holds that
%\begin{equation*}
%\|v_a\|_{\dot{H}^{1/2}}^2 \to \mathcal{S}^{N},~~~~ F_{\mu}(v_a) \to \frac{\mathcal{S}^{N}}{2N} \quad \text{as} \quad a \to 0.
%\end{equation*}
%\item   For any fixed $a >0$, $v_a \in S(a)$ exists for any $\mu >0$ sufficiently small and
%\begin{equation*}
%\|v_a\|_{\dot{H}^{1/2}}^2 \to \mathcal{S}^{N},~~~~ F_{\mu}(v_a) \to \dfrac{\mathcal{S}^N}{2N} \quad \text{as}  \quad \mu \to 0.
%\end{equation*}
%\end{enumerate}
%\end{Thm}

Denote the set of ground states to \eqref{eq1.1}-\eqref{ad1} by $\mathcal{M}_a$. If $u$ solves \eqref{eq1.1} with some $\lambda\in\R$, then $\psi(t,x)=e^{-i\lambda t}u(x)$ satisfies the time-dependent equation
\begin{equation}\label{ab3}
i\partial_t \psi-\sqrt{-\Delta}\psi+\mu|\psi|^{q-2}\psi+|\psi|^{2^*-2}\psi=0, \ \ (t,x)\in \R\times\R^N,
\end{equation}
where $i=\sqrt{-1}$, $\mu\geq0$, $q \!\in\!\big(2,2+\frac{2}{N}\big)$, $2^*:=\frac{2N }{N-1}$. %In particular, if $\mu=0$,
%the local well posed-ness of solutions to \eqref{ab3} is proved in \cite{BGL}.
Suppose the local well posed-ness of solutions to \eqref{ab3} holds for $\mu>0$ (it holds for the case $\mu=0$, see Theorem 1.3 in \cite{BGL}), then we can consider the stability of $\mathcal{M}_a$. The set $\mathcal{M}_a$ is said to be \textbf{stable} under the Cauchy flow of \eqref{ab3} if for any $\varepsilon>0$, there exists $\delta>0$ such that for any $\psi_0 \in H^{1/2}(\R^N )$ satisfying
\[\text{dist}_{H^{1/2}(\R^N )}(\psi_0,\mathcal{M}_a)<\delta,\]
then the solution $\psi(t,\cdot)$ of \eqref{ab3} with $\psi(0,\cdot)=\psi_0$ satisfies
\[ \sup_{t\in [0,T)}\text{dist}_{H^{1/2}(\R^N )}\big(\psi(t,\cdot),\mathcal{M}_a\big)<\varepsilon,\]
where $T$ is the maximal existence time for $\psi(t,\cdot)$. We have
%
%
%
%
%The set $\mathcal{M}_a$ is stable, i.e. for any $\varepsilon > 0$, there exists $\delta >0$ so that if the initial condition $w(0)$ in \eqref{ab3} satisfies
%\begin{equation*}
%\inf_{u\in \mathcal{M}_a}\|w(0)-u\|_{H^{1/2}}<\delta,\ \Rightarrow \ \sup_{t\in [0,\infty)}\inf_{u\in \mathcal{M}_a}\|\psi(t,x)-u\|_{H^{1/2}}<\varepsilon.
%\end{equation*}
%where $\psi(t,x)$ is the solution of \eqref{ab3} corresponding to the initial condition $w(0)$. %and $T_{w}^{max}$ is maximal time of existence.

\begin{Thm}\label{th1.5}
Let $N\ge2$, $2<q<2+\frac{2}{N}$, $\mu>0$ and $a_*=a_*(\mu)>0$ be given in Theorem \ref{th1.1}. Assume that the local well posed-ness of solutions to \eqref{ab3} holds. Then, for any $a\in (0,a_*]$, the set $\mathcal{M}_a$ is compact, up to translation, and it is stable.
\end{Thm}

\noindent\textbf{Remark 1.1.}
In Theorem \ref{th1.1}, the ground state can be taken by a real valued, positive, radially symmetric decreasing function and the mass threshold value $a_*=a_*(\mu)>0$ is explicit, see \eqref{constanta_*}. Moreover, $a_*>0$ can be taken arbitrary large by taking $\mu>0$ small enough. Under the local well posed-ness of \eqref{ab3}, Theorem \ref{th1.5} indicates that a small $L^2$-subcritical term leads to a stabilization of standing waves corresponding to \eqref{ab3}, see e.g. \cite{BGL,DVF}.
Our results complete the works of \cite{BGL,BJ} and extend the results of Jeanjean et al. \cite{JJ,JL}, which studied nonlinear Schr\"{o}dinger equations with combined nonlinearities, to half-wave. We also extend the results of H. J. Luo and Z. T. Zhang \cite{ZL} to Sobolev critical case.   %We extend the main results in J. Bellazzini et al. \cite{BGL} concerning half-wave equation from energy subcritical to energy critical case.
The works of \cite{BGL} focused on half wave in purely mass-supercritical and energy subcritical setting, and showed that the positive ground states are unique up to space shift and phase multiplication.
%稳定性在哪里证明了？
% changes the stability of the model and also
Therefore, our results indicate that the mass subcritical term enriches the set of solutions to \eqref{eq1.1}-\eqref{ad1}.
%Furthermore, the nonlocal equation we studied is more complicated than the classical Schr\"{o}dinger equation considered in \cite{JJ,JL}.

Furthermore, it is worth pointing out that the precise description of the local minimizers $\{u_a\}$ is new in $L^2$-constraint variational problems. Usually, proof of such a precise description depends on a uniqueness of positive solutions to the corresponding limit problem, see \cite{GyWy}. However, this uniqueness result for following limit problem
\begin{equation*}
 \sqrt{-\Delta}u + u = |u|^{q-2}u \ \ \text{in}\ \R^N
\end{equation*}
is still unknown.
We develop another way to derive such precise description of $\{u_a\}$, the main ingredients are more accurate estimates of the ground state energy and Lagrange multiplier together with a mass minimality of the ground state solution, see Lemma \ref{lem5.21} for details.
%In spirit of Jeanjean \cite{JJ}, the ground states are characterized as local minima of $F_\mu$. Lemma \ref{lem3.1}

In the proof of Theorems \ref{th1.1}-\ref{th1.2}, the main difficulty lies in that the term $|u|^{2^*-2}u$ is mass super-critical and energy critical. Although we have $\inf_{u\in S(a)}F_\mu(u)=-\infty$, the mass subcritical term $\mu|u|^{q-2} u$ created a convex-concave geometry of $F_\mu|_{S(a)}$ if $a>0$ is small. So it is reasonable to study a local minimization problem
\begin{equation*}
m(a):=\inf_{u\in V(a)}F_\mu
(u)<0\le\inf_{u\in\partial V(a)}F_\mu(u),
\end{equation*}
where $V(a)\!:=\!\{u\in S(a): \|u\|_{\dot{H}^{1/2}}\!<\!\rho_0 \}$ and $\partial V(a)\!:=\!\{u\in S(a): \|u\|_{\dot{H}^{1/2}}\!=\!\rho_0 \}$. We use $m(a)\!<\!0$ to exclude the vanishing of any minimizing sequences for $m(a)$. Moreover, a scaling argument gives the subadditivity of $m(a)$:
$$  m(a)\le m(\alpha)+m(\sqrt{a^2-\alpha^2}),~~\forall \alpha \in (0,a), $$
which is used to rule out the dichotomy of any minimizing sequences for $m(a)$. Therefore, $m(a)$ is attained by some $u_a\!\in\! V(a)$.
% (see Lemma \ref{lem3.21})
%and Theorem \ref{th1.1} follows

To prove Theorem \ref{th1.1}, we follow some ideas from \cite{LY}. The main ingredient is the refined upper bound of $m(a)$ (see Section 3), i.e.
\begin{equation*}
 m(a)<-K_{N,q}a^{\frac{2q(1-\gamma_q)}{2-q\gamma_q}}<0.
\end{equation*}
This refinement needs to keep the testing functions staying in the admissible set $V(a)$. We overcome this difficulty by utilising the properties of the unique positive radial ground state $Q$ of
\begin{equation}\label{1.8}
 \sqrt{-\Delta}u + u = |u|^{q-2}u \ \ \text{in}\ \R^N,
\end{equation}
where $2<q<2+\frac{2}{N}$. Noticing that the uniqueness of positive solutions to \eqref{1.8} is still unknown, so it brings new difficulties in deriving the precise description of the local minimizers $\{u_a\}$. Fortunately, by the fractional Gagliardo-Nirenberg inequality
\begin{equation*}
 ||u||_{q} \leq C_{opt}(N,q) \|u\|_{\dot{H}^{1/2}}^{\gamma_q} \|u\|_{2}^{1-\gamma_q}, \quad \forall u \in {H}^{1/2}(\mathbb{R}^{N}),
\end{equation*}
%we consider
%\begin{equation*}
%\begin{aligned}
%\inf\limits_{u\in H^{1/2}\setminus\{0\}}\Big[\frac{\|u\|^{\gamma_q}_{\dot{H}^{1/2}}\|u\|^
%{1-\gamma_q}_2}{\|u\|_{q}}\Big]^{\frac{2q}{q-2}}
%=\Big[\frac{\gamma_q}{1-\gamma_q}\Big]^{\frac{q\gamma_q}{q-2}}
%(1-\gamma_q)^{\frac{2}{q-2}}\|Q\|^2_2
%=\Big[\frac{1}{C_{opt}(N,q)}\Big]^{\frac{2q}{q-2}},
%\end{aligned}
%\end{equation*}
and
\begin{equation*}
\begin{aligned}
\frac{1}{C_{opt}(N,q)}=\inf\limits_{u\in H^{1/2}\setminus\{0\}} \frac{\|u\|^{\gamma_q}_{\dot{H}^{1/2}}\|u\|^
{1-\gamma_q}_2}{\|u\|_{q}}
=\Big[\frac{\gamma_q}{1-\gamma_q}\Big]^{\frac{\gamma_q}{2}}
(1-\gamma_q)^{\frac{1}{q}}\|Q\|^{\frac{q-2}{q}}_2,
\end{aligned}
\end{equation*}
we conclude that every ground state of \eqref{1.8} has the least $L^2(\R^N)$ norm among all nontrivial solutions to \eqref{1.8}. This is enough to complete the proof of Theorem \ref{th1.1}, see Section 3.

To prove Theorem \ref{th1.2}, we follow the approach of \cite{JL}. Since $\inf_{u\in S(a)}F_\mu(u)=-\infty$, we can obtain a bounded Palais-Smale sequence by using the Pohozaev constraint approach (see \cite{NsAe,NSaE}). However, it is very difficult to prove the compactness of a Palais-Smale sequence at positive energy level.  Recall that, N. Soave \cite{NSaE} raised an open question on how to obtain the second solution of the Sobolev critical equation \eqref{eq2} since the associated energy functional constrained on the $L^2$-sphere admits a convex-concave geometry. In \cite{JL}, L. Jeanjean and T. T. Le observed that a better energy estimate may contribute to the understanding of this open question and they succeeded in solving this question if $N \ge 4$.
%In fact, the strict inequality of \eqref{b11} guarantees that the functional has a mountain pass geometry. Although $F_{\mu}(u)$ has a mountain-pass geometry on $S(a)$.
Motivated by \cite{JL}, we drive a better energy estimate on the associated mountain pass level (see Proposition \ref{prop3}), i.e.
\begin{equation*}
0<M^0(a)=M(a)<m(a)+\frac{\mathcal{S}^{N}}{2N}.
\end{equation*}
As $m(a)\!<\!0$, the revised upper bound $m(a)+\frac{\mathcal{S}^{N}}{2N}$ is better than the well-known energy threshold $\frac{\mathcal{S}^{N}}{2N}$. This is enough to guarantee the compactness of Palais-Smale sequence at the mountain pass level $M(a)>0$.\\

%We also need to restrict ourselves to $N\ge 2$ in Proposition \ref{prop3}.
%If $N\ge 4\sigma$, Theorem \ref{th1.2} and \ref{th1.3} are also applicable to the general Sobolev space $H^{\sigma}(\R^N)$, $0<\sigma<1$.

%We now underline some of the difficulties that arise in the proof of our main results.

\noindent\textbf{Remark 1.2.}
%, this also implies that the local minimizers $\{u_a\}$ is unique (up to translations) for small mass. %The main ingredient is the refined upper bound of $m(a)$ (see Section 4).
%Such type of limit behavior of local minimizers, we refer to see \cite{LY}.
With slightly modifications, Theorem \ref{th1.1} and Theorem \ref{th1.5} in this paper can be easily extended to the general fractional nonlinear Schr\"{o}dinger equation \eqref{ab1} with $0<\sigma<1$, $\mu>0$, $N>2\sigma$, $p=2_\sigma^\ast=\frac{2N}{N-2\sigma}$ and $2<q<2+\frac{4\sigma}{N}$. Notice that when extending Theorem \ref{th1.2}, if we follow the approach of \cite{JL} to \eqref{ab1}, in our compactness argument, the energy threshold of such compactness $m_\sigma(a)+\frac{\sigma}{N}\mathcal{S}^{N}$ can be carry out only for dimensions $N\ge4\sigma$, see Proposition \ref{prop3}, where $m_\sigma(a)$ is the ground state energy relating to \eqref{ab1}. Turning to directly use the radial superposition of a ground state obtained in Theorem \ref{th1.1} and the Aubin-Talenti bubble as the test function in estimating mountain pass energy threshold, then Theorem \ref{th1.2} can be extended to ${(-\Delta)}^{\sigma}$ for $0<\sigma<1$ in all dimensions $N>2\sigma$. These are motivated by a very recently work \cite{WW}.
%J. C. Wei and Y. Z. Wu and  combining the energy estimates methods in Section 3 of \cite{WW}, we can get the compactness for any $N>2\sigma$. Then Theorem
\\

This paper is organized as follows, in Section 2, we give some preliminary results. The proof of Theorem \ref{th1.1}, Theorem \ref{th1.2} are given in Section 3, Section 4 respectively. \\
%In Section 5, we give the proof of Theorem \ref{th1.2} and Theorem \ref{th1.3}.\\

\textbf{Notations:}~~~~
%$H^{1/2}(\R^N):=\big\{u\in L^2(\R^N),\ \int_{\R^N}\int_{\R^N}\frac{|u(x)-u(y)|^2}{|x-y|^{N+1}}<+\infty\big\}.$
The Hilbert space $H^{\sigma}(\mathbb{R}^{N})$ is defined as
$$
H^{\sigma}(\mathbb{R}^{N}) :=\big\{u \in L^{2}(\mathbb{R}^{N}) :(-\Delta)^{\frac{\sigma}{2}} u \in L^{2}(\mathbb{R}^{N})\big\},
$$
with the inner product and norm given respectively by
$$
(u, v) :=\int_{\mathbb{R}^{N}}(-\Delta)^{\frac{\sigma}{2}} u(-\Delta)^{\frac{\sigma}{2}} v+\int_{\mathbb{R}^{N}} u v,\quad \|u\| :=\Big(\big\|(-\Delta)^{\frac{\sigma}{2}} u\big\|_{2}^{2}+\|u\|_{2}^{2}\Big)^{\frac{1}{2}},
$$
where
$\big\|(-\Delta)^{\frac{\sigma}{2}} u\big\|_{2}^{2} :=\frac{C_{N,\sigma}}{2} \int_{\mathbb{R}^{N}} \int_{\mathbb{R}^{N}} \frac{|u(x)-u(y)|^{2}}{|x-y|^{N+2 \sigma}}$ and $C_{N,\sigma}$ is some positive normalization constant. The space ${\dot{H}}^{\sigma}(\R^{N})$ is defined as the completion of $C_{0}^{\infty}(\R^{N})$ under the norm
$$   ||u||_{\dot{H}^{\sigma}(\R^{N})}= \Big(\int_{\R^{N}} |(-\Delta)^{\sigma/2}u|^2dx\Big )^{\frac{1}{2}}.$$
The dual space of $X$ is denoted by $X'$. %For $\beta \in(0,1)$, $C^{0,\beta}(\R^N)$ denotes the standard H\"{o}lder space on $\R^N$.
$L^{p}=L^{p}(\mathbb{R}^{N})~(1<p\leq\infty)$ is the Lebesgue space with the standard norm $||u||_{p}=\big(\int_{{\mathbb{R}^N}} {{|u(x)|}^p dx}\big)^{{1}/{p}}$. We use $``\rightarrow"$ and $``\rightharpoonup"$ to denote the strong and weak convergence in the related function spaces respectively. $C$ and $C_{i}$ will denote positive constants. %$\langle\cdot,\cdot\rangle$ denote the dual pair for any Banach space and its dual space.

\medskip

\section{Preliminaries}

\setcounter{equation}{0} %The proof of lemmas can be found in the corresponding references.
In this section, we give some preliminary results. The following lemma is the fractional Sobolev embedding.
\begin{Lem}(\cite{EDgE}, Theorem 6.5) \label{lem2.1}
Let $0<\sigma<1$ and $N>2\sigma$. Then there exists a constant $\mathcal{S}=\mathcal{S}(N,\sigma)>0$ such that
\begin{equation*} \label{equ2.1}
\mathcal{S}=\inf _{u \in {\dot{H}}^{\sigma}(\mathbb{R}^{N})\setminus \{0\} }    \frac{\big\|(-\Delta)^{\frac{\sigma}{2}} u\big\|_{2}^{2}}{||u||_{2^{*}_{\sigma}}^{2}},
 %\mathcal{S}||u||_{2^{*}_{\sigma}}^{2} \leq \left\|(-\Delta)^{\frac{\sigma}{2}} u\right\|_{2}^{2}, \qquad \forall u \in {\dot{H}}^{\sigma}(\mathbb{R}^{N}),
\end{equation*}
where $2^{*}_{\sigma}=\frac{2N }{N-2\sigma}$. Moreover, $H^{\sigma}(\mathbb{R}^{N})$ is continuously embedded into $L^{r}(\mathbb{R}^{N})$ for any $2 \leq r \leq 2^{*}_{\sigma}$ and compactly embedded into $L_{\text {loc}}^{r}(\mathbb{R}^{N})$ for every $2 \leq r < 2^{*}_{\sigma}$.
\end{Lem}

We also require the fractional Gagliardo-Nirenberg inequality.
\begin{Lem} \label{lem2.2}
Let $0<\sigma<1$, $N>2\sigma$ and $r\in(2,2^{*}_{\sigma})$. Then there exists a constant $C(N,\sigma,r)$ such that
\begin{equation} \label{equ2.2}
 ||u||_{r} \leq C(N,\sigma,r) \big\|(-\Delta)^{\frac{\sigma}{2}} u\big\|_{2}^{\gamma_r} \left\|u\right\|_{2}^{1-\gamma_r}, \quad \forall u \in {H}^{\sigma}(\mathbb{R}^{N}),
\end{equation}
where $2^{*}_{\sigma}=\frac{2N }{N-2\sigma}$ and $\gamma_r=\frac{N(r-2)}{2r\sigma }$.
\end{Lem}

From Lemma C.2 of \cite{FL}%(or Theorem 2.3 of \cite{FZ})
, the sharp constant $C_{opt}(N,\sigma,r)$ can be obtained by considering
\begin{equation*}
\frac{1}{C_{opt}(N,\sigma,r)}=\inf_{\phi\in H^{\sigma}(\R^N)\setminus\{0\} } \Big\{  \frac{\|(-\Delta)^{\frac{\sigma}{2}}\phi\|^{\gamma_r}_{2}\|\phi\|^{(1-\gamma_r)}_2 }{\|\phi\|_{r}}\Big\}.
\end{equation*}
We can assume that $Q$ is the unique positive radial ground state solution of
\begin{equation}\label{a11}
(-\Delta)^{\sigma} u+u=|u|^{r-2}u\ \ \text{in} \ \R^N,
\end{equation}
see \cite{FL} for details. By the following Pohozaev's identity to \eqref{a11}
\begin{equation*}
\|(-\Delta)^{\sigma/2}Q\|^2_{2}=\gamma_r\|Q\|^{r}_{r}%=\frac{N(r-2)}{4\sigma-(N-2\sigma)(r-2)}\|Q\|^2_2
=\frac{\gamma_r}{1-\gamma_r}\|Q\|^2_2,
\end{equation*}
we get
\begin{equation}\label{a33}
\begin{aligned}
C^{r}_{opt}(N,\sigma,r)&=\Big[\frac{1-\gamma_r}{\gamma_r}\Big]^{\frac{r\gamma_r}{2}}\frac{1}{1-\gamma_r}\frac{1}{\|Q\|^{r-2}_2}.\\
%&=\Big[\frac{4\sigma-(N-2\sigma)(r-2)}{N(r-2)}\Big]^{\frac{N(r-2)}{4\sigma}}\frac{2\sigma r}{4\sigma-(N-2\sigma)(r-2)}\frac{1}{\|Q\|^{r-2}_2}.
\end{aligned}
\end{equation}

%\begin{proof}
%By H\"older's inequality, we have $||u||_{r} \leq ||u||_{2}^{(1-\gamma_r)}||u||_{2^{*}_{\sigma}}^{\gamma_r}$.
%Using (\ref{equ2.1}), we have
%$$ ||u||_{r} \leq ||u||_{2}^{(1-\gamma_r)}\Big(\mathcal{S}^{-\frac{1}{2}}
%||(-\Delta)^{\frac{\sigma}{2}}u||_{2}\Big)^{\gamma_r}= \mathcal{S}^{-\frac{\gamma_r}{2}} %\left\|(-\Delta)^{\frac{\sigma}{2}} u\right\|_{2}^{\gamma_r} \left\|u\right\|_{2}^{(1-\gamma_r)}.$$
%\end{proof}

In the following, we only consider the case $N\ge2$ and $\sigma=\frac{1}{2}$.

\begin{Lem}  \label{lem2.7}(\cite{CW})
Let $N\!\geq\! 2$, $q \!\in\! \big(2,2+\frac{2}{N}\big)$. If $u \!\in\! {H}^{1/2}(\mathbb{R}^{N})$ is a nonnegative weak solution of
\begin{align*}
 \sqrt{-\Delta} u=\lambda u+\mu |u|^{q-2} u+ |u|^{2^*-2}u,
\end{align*}
then the Pohozaev identity holds true
\begin{equation}\label{eq2.5}
0=P_{\mu}(u):=||u||_{\dot{H}^{1/2}}^2
-\mu\gamma_q{||u||}_{q}^{q}-\|u\|^{2^*}_{2^*},
\end{equation}
where $\gamma_q=\frac{N(q-2)}{q}$.
\end{Lem}

To overcome the difficulty that $\inf\limits _{u \in S(a)} F_{\mu}(u)=-\infty$, we introduce the Pohozaev set:
\begin{equation}\label{eq2.4}
\Lambda(a)=\left\{u \in S(a) : P_{\mu}(u)=0\right\}.
\end{equation}
%where
%\begin{equation}  \label{eq2.5}
%P_{\mu}(u) =:||u||_{\dot{H}^{1/2}}^2
%-\mu\gamma_q{||u||}_{q+1}^{q+1}-\|u\|^4_4
%\end{equation}
%for $\gamma_q=\frac{2(q-1)}{q+1}$. \Lambda
As a consequence of Lemma \ref{lem2.7}, we known that any nonnegative critical point of $F_{\mu}|_{S(a)}$ stays in $\Lambda(a)$. The properties of $\Lambda(a)$ are related to the minimax structure of $F_{\mu}|_{S(a)}$, and in particular to the behavior of $F_{\mu}$ with respect to dilations preserving the $L^2$-norm. To be more precise, for $u \in S(a)$ and $t\in \mathbb{R}$, let
\begin{equation*} \label{eq1.12}
 t\star u(x) :=t^{\frac{N}{2}} u\left(t x\right),~~~~\mbox{for}~~~~\mbox{a.e.}~~~~x \in \mathbb{R}^{N}.
\end{equation*}
It results that $t\star u \in S(a)$, and hence it is natural to study the fiber map
\begin{equation} \label{eq1.13}
\Psi_{u}(t) :=F_{\mu}(t\star u)=\frac{t}{2} ||u||_{\dot{H}^{1/2}}^2-\mu \frac{t^{\frac{q\gamma_q}{2}}}{q} {||u||}_{q}^{q}-\frac{ t^{\frac{2^*}{2}}}{2^*}\|u\|^{2^*}_{2^*}.
\end{equation}
%We shall see that critical point of $\Psi_{u}(s)$ allow to project a function on $\mathcal{P}_a$. Thus, monotonicity and convexity properties of $\Psi_{u}(s)$ strongly affects the structure of $\mathcal{P}_{a}$ (and in turn the geometry of $F_{\mu}|_{S(a)}$ ), and also have a strong impact on properties of equation (\ref{eq1.1}).
%In this direction, let us consider the decomposition of $\mathcal{P}_{a, \mu}$ into the disjoint union $\mathcal{P}_{a, \mu}=\mathcal{P}_{+}^{a, \mu}\cup \mathcal{P}_{-}^{a, \mu} $, where
%$$\mathcal{P}_{+}^{a, \mu} :=\left\{u \in \mathcal{P}_{a, \mu} : ||u||_{\dot{H}^{1/2}}^2>
%\frac{2\mu(q-1)^2}{q+1}{||u||}_{q+1}^{q+1}+2\|u\|^4_4 \right\}=\left\{u \in \mathcal{P}_{a, \mu} :\left(\Psi_{u}^{\mu}\right)^{\prime \prime}(0)>0\right\}$$
%$$\mathcal{P}_{-}^{a, \mu} :=\left\{u \in \mathcal{P}_{a, \mu} : ||u||_{\dot{H}^{1/2}}^2<
%\frac{2\mu(q-1)^2}{q+1}{||u||}_{q+1}^{q+1}+2\|u\|^4_4 \right\}=\left\{u \in \mathcal{P}_{a, \mu} :\left(\Psi_{u}^{\mu}\right)^{\prime \prime}(0)<0\right\}.$$

For $u \in S(a)$, $t \in (0,\infty)$ and the fiber $\Psi_{u}$ introduced in (\ref{eq1.13}), we have
\begin{equation} \label{equa2.9}
\Psi'_{u}(t)= \frac{1}{2}||u||_{\dot{H}^{1/2}}^2
-\frac{\mu\gamma_q}{2}t^{\frac{q\gamma_q}{2}-1}{||u||}_{q}^{q}-\frac{t^{\frac{2^*}{2}-1}}{2}\|u\|^{2^*}_{2^*} =\frac{1}{2t}P_{\mu}(t \star u),
\end{equation}
where $P_{\mu}$ is defined by (\ref{eq2.5}). From (\ref{equa2.9}), we can see immediately that:

\begin{Cor}\label{cor2.1}
Let $u \in S(a)$. Then $t \in (0,\infty)$ is a critical point for $\Psi_{u}$ if and only if $ t\star u \in \Lambda(a)$.
\end{Cor}

In particular, $u \in \Lambda(a)$ if and only if $1$ is a critical point of $\Psi_{u}$. For future convenience, we also recall that the map $(t, u) \in (0,\infty) \times H^{1/2}(\R^N) \mapsto t\star u \in H^{1/2}(\R^N)$ is continuous (The proof is similar to Lemma 3.5 in \cite{bTEv}).

We also need the following result, where $T_uS(a)$ denotes the tangent space to $S(a)$ in $u$.
\begin{Lem} \label{lem2.8}(\cite{bTEv}, Lemma 3.6)
For $u \in S(a)$ and $t\in (0,\infty)$ the map
$$
T_{u}S(a) \rightarrow T_{t\star u}S(a), \quad \varphi \mapsto t\star \varphi
$$
is a linear isomorphism with inverse $\psi \mapsto \frac{1}{t}\star\psi$.
\end{Lem}

%\begin{proof}
%It is similar to the proof of Lemma 3.6 in \cite{bTEv}.
%\end{proof}

\begin{Lem}\label{lem2.9}(\cite{Py})
Let $u^*$ denote the symmetric decreasing rearrangement of $u$. Then we have
\begin{equation*}
\int_{\R^N}\int_{\R^N}\frac{|u^*(x)-u^*(y)|^2}{|x-y|^{N+1}}\le\int_{\R^N}\int_{\R^N}\frac{|u(x)-u(y)|^2}{|x-y|^{N+1}}.
\end{equation*}
\end{Lem}

\section{The proof of Theorem 1.1}

In this section, we deal with the case $2<q<2+\frac{2}{N}$, $\mu>0$ and prove Theorem \ref{th1.1}. %Replacing $C^{q}_{opt}$ (see \eqref{a33}) by $C^{q+1}_q$.
By the fractional Gagliardo-Nirenberg inequality and Sobolev inequality, we get
\begin{equation}\label{a112}
\begin{aligned}
F_\mu(u)&=\frac{1}{2}\|u\|^2_{\dot{H}^{1/2}}-\frac{\mu}{q}\|u\|^{q}_{q}-\frac{1}{2^*}\|u\|^{2^*}_{2^*}\\
&\ge\frac{1}{2}\|u\|^2_{\dot{H}^{1/2}}-\frac{\mu}{q}C^{q}_{opt}\|u\|^{q\gamma_q}_{\dot{H}^{1/2}}\|u\|^{q(1-\gamma_q)}_2-\frac{1}{2^*\mathcal{S}^\frac{2^*}{2}}\|u\|^{2^*}_{\dot{H}^{1/2}}\\
&=h_a(\|u\|_{\dot{H}^{1/2}}).
\end{aligned}
\end{equation}
Let
$$f(a,\rho)=h_a(\rho):=\frac{1}{2}\rho^2-\frac{\mu}{q}C^{q}_{opt}\rho^{q\gamma_q}a^{q(1-\gamma_q)}-\frac{1}{2^*\mathcal{S}^\frac{2^*}{2}}  \rho^{2^*}.$$
Since $\mu>0$ and $q\gamma_q\!<\!2\!<\!2^*$, we have that $h_a(0^+)\!=\!0^{-}$ and $h_a(+\infty)\!=\!-\infty $. Denote
\begin{equation} \label{constanta_*}
a_*\!:=\!\Big\{\frac{q(2^*\!-\!2)}{2\mu C^{q}_{opt}(2^*\!-\!q\gamma_q)}
\Big[\!\frac{(2\!-\!q\gamma_q)2^*S^{\frac{2^*}{2}}}{2(2^*\!-\!q\gamma_q)}\Big]^{\frac{2-q\gamma q}{2^*-2}} \Big\}^{\frac{1}{q(1-\gamma q)}}.
\end{equation}

\begin{Lem}\label{lem3.1}
Let $a\!>\!0$, $\mu>\!0$, $2\!<\!q\!<\!2+\frac{2}{N}$,  $\rho_0\!=\Big[\!\frac{(2\!-\!q\gamma_q)2^*S^{\frac{2^*}{2}}}{2(2^*\!-\!q\gamma_q)}\Big]^{\frac{1}{2^*-2}}$ and $\tilde{\rho}_a:=\Big[\frac{2\mu}{q}C^{q}_{opt} a^{q(1-\gamma_q)}\Big]^{\frac{1}{2-q\gamma_q}}$. Then\\
\indent (i) If $a\!<\!a_*$, $h_a(\rho)$ has a local minimum at negative level and a global maximum at positive level; moreover, there exist $R_0=R_0(a)$ and $R_1=R_1(a)$ such that
$$0<\tilde{\rho}_a< R_0 < \frac{a}{a_*} {\rho}_{0}<{\rho}_0< R_1,~~~~~~~~h_a(R_0)=0=h_a(R_1),~~~~~~~~h_a(\rho)\!>\!0 \Leftrightarrow \rho\!\in\! (R_0,R_1);$$
\indent (ii) If $a\!=\!a_*$, $h_{a_*}(\rho)$ has a local minimum at negative level and a global maximum at level $0$; moreover, we have
$h_{a_*}(\rho_{a_*})=0,~~~~~~~~h_{a_*}(\rho)\!<\!0 \Leftrightarrow \rho\!\in\!(0, {\rho}_{a_*}) \cup ({\rho}_{a_*},+\infty)$.
\end{Lem}
\begin{proof}

(i) We first prove that $h_a$ has exactly two critical points. In fact,
$$
h'_a(\rho)=0 \Longleftrightarrow \phi_a(\rho)=\mu\gamma_q C^{q}_{opt}a^{q(1-\gamma_q)}, ~~~~\mbox{with}~~~~\phi_a(\rho)=\rho^{2-q\gamma_q}-\frac{1}{\mathcal{S}^\frac{2^*}{2}}\rho^{2^*-q\gamma_q}.
$$
We have, $\phi_a(\rho) \nearrow$ on $[0,\bar{\rho})$, $\searrow$ on $(\bar{\rho},+\infty)$, where $\bar{\rho}\!=\Big[\!\frac{(2\!-\!q\gamma_q)S^{\frac{2^*}{2}}}{2^*\!-\!q\gamma_q}\Big]^{\frac{1}{2^*-2}}$.
Since $q\gamma_q\!<\!2$, we get $\max\limits_{\rho\geq0}\phi_a(\rho)\!=\!\phi_a(\bar{\rho})
\!=\!\frac{2^*\!-\!2}{2^*\!-\!q\gamma_q}\bar{\rho}^{2\!-\!q\gamma_q}
\!>\!\mu \gamma_q C^{q}_{opt}a^{q(1-\gamma_q)}$ if and only if
\begin{equation*} \label{c^*}
a\!<\!a^*\!:=\!\Big\{\frac{2^*\!-\!2}{\mu\gamma_q C^{q}_{opt}(2^*\!-\!q\gamma_q)}
\Big[\!\frac{(2\!-\!q\gamma_q)S^{\frac{2^*}{2}}}{2^*\!-\!q\gamma_q}\Big]^{\frac{2-q\gamma q}{2^*-2}} \Big\}^{\frac{1}{q(1-\gamma q)}}.
\end{equation*}
By $\phi_a(0^+)\!=\!0^+$ and $\phi_a(+\infty)\!=\!-\infty$, we have that $h_a$ has exactly two critical points if $a\!<\!a^*$.

Notice that
$$
h_a(\rho)>0 \Longleftrightarrow \psi_a(\rho)>\frac{\mu}{q}C^{q}_{opt}a^{q(1-\gamma_q)}, ~~~~\mbox{with}~~~~\psi_a(\rho)=\frac{1}{2}\rho^{2-q\gamma_q}-\frac{1}{2^*\mathcal{S}^\frac{2^*}{2}}\rho^{2^*-q\gamma_q}.
$$
It is not difficult to check that $\psi_a(\rho) \nearrow$ on $[0,\rho_0)$, $\searrow$ on $(\rho_0,+\infty)$, where $\rho_0\!=\Big[\!\frac{(2\!-\!q\gamma_q)2^*S^{\frac{2^*}{2}}}{2(2^*\!-\!q\gamma_q)}\Big]^{\frac{1}{2^*-2}}$.
Since $q\gamma_q\!<\!2$, we have $\max\limits\limits_{\rho\geq0}\psi_a(\rho)\!=\!\psi_a(\rho_0)
\!=\!\frac{2^*\!-\!2}{2(2^*\!-\!q\gamma_q)}\rho_0^{2\!-\!q\gamma_q}
\!>\!\frac{\mu}{q}C^{q}_{opt}a^{q(1-\gamma_q)}$ provided
\begin{equation*}
a\!<\!a_*\!:=\!\Big\{\frac{q(2^*\!-\!2)}{2\mu C^{q}_{opt}(2^*\!-\!q\gamma_q)}
\Big[\!\frac{(2\!-\!q\gamma_q)2^*S^{\frac{2^*}{2}}}{2(2^*\!-\!q\gamma_q)}\Big]^{\frac{2-q\gamma q}{2^*-2}} \Big\}^{\frac{1}{q(1-\gamma q)}}.
\end{equation*}
We have $h_a(t)\!>\!0$ on an open interval $(R_0,R_1)$ if and only if $a\!<\!a_*$. It suffices to show that $\frac{\log{x}}{x-1}$ is a monotone decreasing function of $x$, which follows
$$\big(\frac{2^*}{2}\big)^{\frac{2-q\gamma_q}{2^*-2}}\frac{q\gamma_q}{2}<1\Leftrightarrow a_*<a^*,$$
(see Lemma 5.2 of \cite {NsAe}). Combined with $h_a(0^+)=0^{-}$ and $h_a(+\infty)=-\infty$, we see that $h_a$ has a local minimum point at negative level in $(0,R_0)$ and a global maximum point at positive level in $(R_0,R_1)$.
We also deduce that $\psi_a( {\rho}_{0})=\psi_{a_*}( {\rho}_{0})=\frac{\mu}{q}C^{q}_{opt}a^{q(1-\gamma_q)}_*>\frac{\mu}{q}C^{q}_{opt}a^{q(1-\gamma_q)}$ and
\begin{align*}
 \psi_a\big(\frac{a}{a_*} {\rho}_{0}\big)
&=\frac{{\rho}_{0}^{2\!-\!q\gamma_q}}
{2}\Big[1-\frac{ 2-q\gamma_q}{2^*-q\gamma_q}
\big(\frac{a}{a_*}\big)^{2^*-2}\Big]\big(\frac{a}{a_*}\big)^{2-q\gamma_q}
>\frac{2^*-2}{2(2^*-q\gamma_q)}{\rho}_{0}^{2\!-\!q\gamma_q}\big(\frac{a}{a_*}\big)^{2-q\gamma_q} \\
&=\frac{\mu}{q}C^{q}_{opt}a^{q(1-\gamma_q)}_*\big(\frac{a}{a_*}\big)^{2-q\gamma_q}=\frac{\mu}{q}C^{q}_{opt}a^{q(1-\gamma_q)}
\big(\frac{a_*}{a}\big)^{q-2} \\
&>\frac{\mu}{q}C^{q}_{opt}a^{q(1-\gamma_q)},
\end{align*}
which gives $h_a(\rho_0)\!>\!0$, $h_a(\frac{a}{a_*} \rho_0)\!>\!0$, and hence $0\!<\! R_0 \!<\!\frac{a}{a_*} {\rho}_{0}\!<\! {\rho}_{0}\!< \! R_1$.
Finally, the fact that $h_a(\rho)\!\leq\!g_a(\rho)=\frac{1}{2}\rho^{2}-\frac{\mu}{q}C^{q}_{opt} a^{q(1-\gamma_q)}\rho^{q\gamma_q}$ leads to $R_0\!>\!\tilde{\rho}_a$, where $\tilde{\rho}_a:=\Big[\frac{2\mu}{q}C^{q}_{opt} a^{q(1-\gamma_q)}\Big]^{\frac{1}{2-q\gamma_q}}$.

(ii) Similar to the proof of (i), we deduce that
$$
R_0\big(a_*\big)= {\rho}_{a_*}=\rho_{0}=R_1\big(a_*\big),~~~~~~~~\psi_{a_*}( {\rho}_{a_*})=\frac{\mu}{q}C^{q}_{opt}a^{q(1-\gamma_q)}_*,~~~~~~~~\phi_{a_*}({\rho}_{a_*})=\mu \gamma_qC^{q}_{opt}a^{q(1-\gamma_q)}_*.
$$
\end{proof}

Let $a_*>0$ and $\rho_0 >0$ be given in Lemma \ref{lem3.1}. Note that by Lemma \ref{lem3.1}, we have that $f(a_*, \rho_0)=h_{a_*}(\rho_0)= 0$ and $f(a, \rho_0)=h_a(\rho_0) \ge 0$ for all $a \in (0, a_*]$.
Define
\begin{align*}
B_{\rho_0} := \{u \in H^{1/2}: \|u\|_{\dot{H}^{1/2}} < \rho_0\} \quad \text{and} \quad V(a) := S(a) \cap B_{\rho_0}.
\end{align*}
Thus, we can define the local minimization problem:  for any $a \in (0, a_*]$,
\begin{equation*}
m(a) := \inf_{u \in V(a)} F_{\mu}(u).
\end{equation*}
In order to prove that $F_{\mu}$ has a local strict minimum at negative level, we give some lemmas.

\begin{Lem}\label{lem3.3} If $a \in (0, a_*]$, then the following properties hold,
\begin{enumerate}	
\item
$$m(a) = \inf_{u \in V(a)} F_{\mu}(u) < 0 \leq  \inf_{u \in \partial V(a)}F_{\mu}(u).$$	
\item If $m(a)$ is reached, then any ground state is contained in $V(a)$.
\end{enumerate}
\end{Lem}

\begin{proof}
 (1) If $u \in \partial V(a)$, by \eqref{a112}, we obtain
\begin{align*}
F_\mu(u)&\ge f(\|u\|_2,\|u\|_{\dot{H}^{1/2}})=f(a,\rho_0)\ge f(a_*,\rho_0)=0.
\end{align*}
For any fixed $u \in S(a)$,
%\begin{align*}
%t\star u(x) := t^{\frac{N}{2}} u(t x),\ \ \text{for}\ a.e. \ x\in \R^N.
%\end{align*}
clearly $t\star u \in S(a)$ for any $t \in (0, \infty)$. From \eqref{eq1.13},
\begin{align*}
\Psi_{u}(t) :=F_{\mu}(t\star u)=\frac{t}{2} ||u||_{\dot{H}^{1/2}}^2-\mu \frac{t^{\frac{q\gamma_q}{2}}}{q} {||u||}_{q}^{q}-\frac{ t^{\frac{2^*}{2}}}{2^*}\|u\|^{2^*}_{2^*}.
\end{align*}
Since $\frac{q\gamma_q}{2}<1$ and $\frac{2^*}{2}>1$, it implies that $\Psi_u(t)\to 0^-$ as $t\to 0$. Therefore, there exists $t_0>0$ small enough such that $\|t_0\star u\|^2_{\dot{H}^{1/2}}=t_0\|u\|^2_{\dot{H}^{1/2}}<\rho_0$ and $ F_{\mu}(t_0\star u)=\Psi_u(t_0)<0$. Thus, we know that $m(a)<0$.

(2) %By Lemma \ref{lem2.7}, for any $u\in \mathcal{S}_a$, there exists $v\in \mathcal{S}_a$ such that $u=v_t$, $\|v\|^2_{\dot{H}^{1/2}}=1$ and $t\in (0,\infty)$.
By Corollary \ref{cor2.1}, %it is easy to verify that any critical point of $\Psi_\mu$ belongs to $\Lambda(a)$. Conversely,
if $u\in \Lambda(a)$, we get $\Psi'_u(1)=0$. We deduce from \eqref{equa2.9} that if $w \in S(a)$ is a ground state solution there exist a $v \in S(a)$ and a $t_w=\frac{\|w\|^2_{\dot{H}^{1/2}}}{\|v\|^2_{\dot{H}^{1/2}}}\in (0,\infty)$ such that $w(x)=t_w\star v=t^{\frac{N}{2}}_w v(t_wx)\in S(a)$, we have
$$\Psi_v(t_w)=F_\mu(t_w\star v)=F_\mu(w)<0,\quad \Psi'_v(t_w)=P_\mu(t_w\star v)=P_\mu(w)=0 .$$
%It follows that $\|v\|_{\dot{H}^{1/2}}^2\le \|v\|_{\dot{H}^{1/2}}^2$ and a $t_0 \in (0, \infty)$ such that $w= v_{t_0}$, $F_{\mu}(w) = \Psi_v(t_0)$ and $\Psi_{v}'(t_0)=0$.
Namely, $t_w \in (0, \infty)$ is a zero of the function $\Psi_{v}'.$

Since $\Psi_v(t) \to 0^-$, $\|t\star v\|_2 \to 0,$ as $t \to 0$ and $\Psi_v(t) = F_{\mu}(t\star v) \geq 0$ when $t\star v \in \partial V(a) = \{u \in S(a) : ||u||_{\dot{H}^{1/2}}^2 = \rho_0 \}$.
$\Psi_v(t)$ has a local minimum point $t_1 >0$, ${t_1}\star v \in V(a)$ and $F({t_1}\star v) = \Psi_v(t_1) <0$.
Also, from $\Psi_v(t_1)<0$, $\Psi_v(t) \geq  0$ when $t\star v \in \partial V(a)$ and $\Psi_v(t) \to - \infty$ as $t \to  \infty$, $\Psi_v(t)$ has a global maximum point of $t_2 >t_1$. Since ${t_2}\star v$ satisfies $F({t_2}\star v)= \Psi_v(t_2) \geq 0$, we have that $m(a) \leq F({t_1}\star v) < F({t_2}\star v)$. In particular, since $m(a)$ is reached, ${t_2}\star v$ cannot be a ground state. Since $\Psi''_v(t)=0$ has a unique solution, we see that $\Psi'_v(t)$ has at most two zeros. Thus the conclusion (2) follows from these facts above.

\end{proof}

%In the following, we can give some properties of $m(a)$.
\begin{Lem} \label{lem3.4}
If $a \in (0,a_*]$, the following properties hold.
\begin{enumerate}
		\item $a\mapsto m(a)$ is a continuous map.
		\item  We have for all $\alpha \in (0,a)$ : $m(a) \leq m(\alpha) + m(\sqrt{a^2-\alpha^2})$ and if $m(\alpha)$ or $m(\sqrt{a^2-\alpha^2})$ is reached  then  the inequality is strict.
	\end{enumerate}
\end{Lem}
\begin{proof}
$(1)$\ In order to prove the continuity consider a sequence $a_n\in (0,a_*]$ such that $a_n\to a$ as $n\to\infty$. From Lemma \ref{lem3.3} (1),
%the definition of $m(a_n)$ and since $m(a_n) <0$, see
 for any $\varepsilon >0$ sufficiently small, there exists $u_n \in V(a_n)$ such that
\begin{align}\label{a2}
F_{\mu}(u_n) \leq m(a_n) + \varepsilon \quad \mbox{and} \quad F_{\mu}(u_n) <0.
\end{align}
Let $\tilde{u}_n := \frac{a}{a_n} u_n$, then $\tilde{u}_n \in S(a)$.

We claim that $\tilde{u}_n \in V(a)$. In fact, if $a_n \geq a$, then
\begin{align*}
\|\tilde{u}_n\|_{\dot{H}^{1/2}} = \frac{a}{a_n} \|u_n\|_{\dot{H}^{1/2}} \leq \|u_n\|_{\dot{H}^{1/2}} < \rho_0.
\end{align*}
If $a_n<a$, by Lemma \ref{lem3.1}, we have $f(a_n, \rho) \geq 0$ for any $\rho\in [\frac{a_n}{a}\rho_0,\rho_0]$.
Therefore, we deduce from Lemma \ref{lem3.1} and \eqref{a2} that $f(a_n, \|u_n\|_{\dot{H}^{1/2}})<0$, then $\|u_n\|_{\dot{H}^{1/2}}<\frac{a_n}{a}\rho_0$ and
\begin{align*}
\|\tilde{u}_n\|_{\dot{H}^{1/2}}=\frac{a}{a_n} \|u_n\|_{\dot{H}^{1/2}} < \frac{a}{a_n} \frac{a_n}{a} \rho_0 = \rho_0.
\end{align*}
This completes the proof.

Since $\tilde{u}_n \in V(a)$, we have
\begin{align*}
m(a) \leq F_{\mu}(\tilde{u}_n) = F_{\mu}(u_n) + [F_{\mu}(\tilde{u}_n) - F_{\mu}(u_n)],
\end{align*}
where
\begin{equation*}
F_{\mu}(\tilde{u}_n) - F_{\mu}(u_n) = \frac{1}{2}(\frac{a^2}{a^2_n}-1)\|u_n\|_{\dot{H}^{1/2}}^2 - \frac{\mu}{q} \big[ (\frac{a}{a_n})^{q}-1\big]  \|u_n\|_{q}^{q} -  \frac{1}{2^*}\big[(\frac{a}{a_n})^{2^*} - 1 \big] \|u_n\|_{2^*}^{2^*}.
\end{equation*}
As $\|u_n\|_{\dot{H}^{1/2}} <\rho_0$, $\{u_n\}$ is uniformly bounded in $H^{1/2}(\R^N)$.
%$\|u_n\|_{q+1}^{q+1}$ and $\|u_n\|_4^4$ are uniformly bounded,
Then, we have
\begin{equation}\label{a3}
m(a) \leq F_{\mu}(\tilde{u}_n) = F_{\mu}(u_n) + o_n(1) \ \ \text{as} \ \  n \to \infty.
\end{equation}
Combining \eqref{a2} and \eqref{a3}, we have
\begin{align*}
m(a) \leq m(a_n) + \varepsilon + o_n(1).
\end{align*}

Similarly, there exists a $u \in V(a)$ such that
\begin{equation*}
F_{\mu}(u) \leq m(a) + \varepsilon \quad \mbox{and} \quad F_{\mu}(u) <0.
\end{equation*}
Let $\bar{u}_n := \frac{a_n}{a}u$ and hence $\bar{u}_n \in S(a_n)$.
Clearly, $\|u\|_{\dot{H}^{1/2}}<\rho_0$ and $a_n \to a$ imply $\|\bar{u}_n\|_{\dot{H}^{1/2}} < \rho_0$ for $n$ large enough,
so that $\bar{u}_n \in V(a_n)$. Also, $F_{\mu}(\bar{u}_n) \to F_\mu(u)$ as $n\to +\infty$. We thus have
\begin{equation*}
m(a_n) \leq F_\mu(\bar{u}_n) = F_\mu(u) + [F_{\mu}(\bar{u}_n) - F_{\mu}(u)] \leq m(a) + \varepsilon + o_n(1).
\end{equation*}
Hence, since $\varepsilon > 0$ is arbitrary, we obtain that $m(a_n) \to m(a)$ as $n\to\infty$.

(2) By Lemma \ref{lem3.3} (1), fixed $\alpha\in(0,a)$, for any $\varepsilon >0$ small enough, there exists a $u \in V(\alpha)$ such that
\begin{align}\label{a5}
F_{\mu}(u) \leq m(\alpha) + \varepsilon \quad \text{and} \quad F_{\mu}(u) <0.
\end{align}
In view of Lemma \ref{lem3.1},  $\displaystyle f(\alpha, \rho) \geq 0$ for any $\rho \in [\frac{\alpha}{a} \rho_0, \rho_0]$. Combining Lemma \ref{lem3.1} with \eqref{a5}, we have
\begin{equation}\label{a6}
||u||_{\dot{H}^{1/2}} < \frac{\alpha}{a} \rho_0.
\end{equation}
Setting $v = \theta u$ and $1<\theta<\frac{a}{\alpha}$, we have $||v||_2 = \theta ||u||_2 = \theta \alpha$ and
$||v||_{\dot{H}^{1/2}} = \theta ||u||_{\dot{H}^{1/2}} < \rho_0$, which implies $v \in V(\theta \alpha)$.
Therefore
\begin{equation}\label{a51}
\begin{aligned}
m(\theta \alpha) \leq F_{\mu}(v) &= \frac{\theta^2}{2} ||u||_{\dot{H}^{1/2}}^2 - \frac{\mu\theta^{q}}{q}  \|u\|_{q}^{q} - \frac{\theta^{2^*}}{2^*}  \|u\|_{2^*}^{2^*}  \\
&< \theta^2 F_{\mu}(u) \leq \theta^2 (m(\alpha) + \varepsilon).
\end{aligned}
\end{equation}
Since $\varepsilon > 0$ is arbitrary, we have that  $m(\theta \alpha) \leq \theta^2 m(\alpha)$.
If $m(\alpha)$ is reached then we can let $\varepsilon = 0$ in \eqref{a51} and thus the strict inequality follows.

Next, we show that $m(a)\le m(\alpha)+m(\sqrt{a^2-\alpha^2})$ for $\alpha \in (0,a)$.
Indeed, then we have
\begin{align*}
m(a)&=\frac{\alpha^2}{a^2}m(a)+\frac{a^2-\alpha^2}{a^2}m(a)\\
&=\frac{\alpha^2}{a^2}m(\frac{a}{\alpha}\alpha)+
\frac{a^2-\alpha^2}{a^2}m(\frac{a}{\sqrt{a^2-\alpha^2}}\sqrt{a^2-\alpha^2})\\
&\le m(\alpha)+m(\sqrt{a^2-\alpha^2}),
\end{align*}
%\begin{align*}
%m(a)=m\big(\frac{a}{\alpha}\alpha\big)
%&\le\frac{a^2}{\alpha^2}m(\alpha)
%=m(\alpha)+\frac{a^2-\alpha^2}{\alpha^2}m(\alpha)\\
%&=m(\alpha)+\frac{a^2-\alpha^2}{\alpha^2}m\big((a-\alpha)
%\frac{\alpha}{a-\alpha}\big)\\
%&\le m(\alpha)+m(a-\alpha),
%\end{align*}
with a strict inequality if $m(\alpha)$ is reached.
\end{proof}

Define
\begin{equation}\label{Ma}
\mathcal{M}_a:=\big\{u\in V(a) \big| F_{\mu}(u)=m(a)\big\}.
\end{equation}
Obviously, $\mathcal{M}_a$ contains all ground state solutions of \eqref{eq1.1}-\eqref{ad1}. Moreover, we have the following Lemma.

\begin{Lem}\label{lem3.5}
Suppose $2<q<2+\frac{2}{N}$ and $a\in (0,a_*]$, then $m(a)$ can be achieved by some function $u_a\in V(a)$.
%$\mathcal{M}_a$ is compact in $H^{1/2}(\R^N)$.
\end{Lem}
\begin{proof}
For any $a\in (0,a_*]$, let $\{u_n\}$ be a minimizing sequence for $m(a)$ such that $u_n\in S(a)$ and $F_{\mu}(u_n)\to m(a)$. Then, there exist a $\nu>0$ and $\{y_n\}$ such that
\begin{equation*}
\int_{B(y_n,R)}|u_n|^2\geq\nu>0,\quad  \text{for some}\ R>0.
\end{equation*}
If not, by Lions Lemma \cite{WM}, we have $\|u_n\|^q_{q}\to 0$ as $n\to \infty$. Using the Sobolev inequality, we deduce that
\begin{equation}\label{a61}
\begin{aligned}
F_\mu(u_n)&=\frac{1}{2}||u||_{\dot{H}^{1/2}}^2 - \frac{1}{2^*} \|u\|_{2^*}^{2^*}+o_n(1)\\
          &\ge\frac{1}{2}||u||_{\dot{H}^{1/2}}^2 - \frac{1}{2^*\mathcal{S}^{\frac{2^*}{2}}}||u||_{\dot{H}^{1/2}}^{2^*}+o_n(1)\\
          &\ge||u||_{\dot{H}^{1/2}}^2\Big[\frac{1}{2}-\frac{1}{2^*\mathcal{S}^{\frac{2^*}{2}}}\rho^{2^*-2}_0 \Big]+o_n(1).
\end{aligned}
\end{equation}
By Lemma \ref{lem3.1}, $f(a_*,\rho_0)=0$, we get
$$\beta_0=\Big[\frac{1}{2}-\frac{1}{2^*\mathcal{S}^{\frac{2^*}{2}}}\rho^{2^*-2}_0\Big]=\frac{\mu C^{q}_{opt}}{q} \rho^{q\gamma_q-2}_0a^{q(1-\gamma_q)}_*>0.$$
This contradicts with $m(a)<0$.

Then there exists $u_a\in H^{1/2}(\R^N)$ such that $u_n(x-y_n)\rightharpoonup u_a\neq 0$ in $H^{1/2}(\R^N)$.  Let $w_n(x):=u_n(x-y_n)-u_a$, by Br\'ezis-Lieb Lemma, we have
\begin{equation*}
\|w_n\|^2_2=\|u_n(x-y_n)\|^2_2-\|u_a\|^2_2=a^2-\|u_a\|^2_2+o_n(1).
\end{equation*}
Moreover,
\begin{equation*}
F_{\mu}(u_n)=F_{\mu}(u_n(x-y_n))=F_{\mu}(w_n)+F_{\mu}(u_a)+o_n(1).
\end{equation*}

Next, we claim that
\begin{align*}
	\|w_n\|_2^2 \to 0,\quad \text{as} \quad n \to \infty.
\end{align*}

In fact, if $a_0:=\|w_n\|_2>0$, we have $\|w_n\|_2\le a$ and $\|w_n\|_{\dot{H}^{1/2}}\le ||u_n||_{\dot{H}^{1/2}}<\rho_0$. Hence, we get that $w_n \in V(a_0)$ and $F_{\mu}(w_n) \geq m(a_0)$.
Recalling that $F_{\mu}(u_n) \to m(a)$, we have
\begin{align*}
	m(a) = F_{\mu}(w_n) + F_{\mu}(u_a) + o_n(1) \geq m(a_0) + F_{\mu}(u_a) + o_n(1).
\end{align*}
From Lemma \ref{lem3.4}, the map $a \mapsto m(a)$ is continuous, we deduce that
\begin{align}\label{a7}
m(a) \geq m(a_0) + F_{\mu}(u_a).
\end{align}
We also have that $u_a \in V(\|u_a\|_2)$ by the weak limit. It follows that $F_{\mu}(u_a) \geq m(\|u_a\|_2)$.
If $F_{\mu}(u_a) > m(\|u_a\|_2)$, then it follows from \eqref{a7} that
	\begin{align*}
	m(a) > m(a_0) + m(\sqrt{a^2-a^2_0}) \geq  m(a),
	\end{align*}
which is a contradiction. Thus we have $F_{\mu}(u_a) = m(\sqrt{a^2-a^2_0})$, namely $u_a$ is a local minimizer on $V(a-a_0)$. Therefore, we deduce from Lemma \ref{lem3.4} and \eqref{a7} that
\begin{align*}
m(a) \geq m(a_0) + F_{\mu}(u_a) = m(a_0) + m(\sqrt{a^2-a^2_0}) >  m(a),
\end{align*}
which is impossible. Thus, $a_0=0$, the claim follows.
	
We will prove that $w_n \to 0$ in $H^{1/2}(\R^N)$. From the above claim, we have $\|u_a\|_2=a$ and $\|u_a\|_{\dot{H}^{1/2}}\le ||u_n||_{\dot{H}^{1/2}}<\rho_0$, for $n$ large enough. Then %$F_{\mu}(u_a) \geq m(a)$ and
\begin{equation*}
	F_{\mu}(u_n) = F_{\mu}(u_a) + F_{\mu}(w_n) + o_n(1) \to m(a).
\end{equation*}
We have that  $F_{\mu}(u_a) \geq m(a)$, and hence $F_{\mu}(w_n) \leq o_n(1)$. By using the fractional Gagliardo-Nirenberg inequality and $||w_n||_2^2 \to 0$, we also have $||w_n||_q^q \to 0$.
Thus \eqref{a61} implies that
\begin{equation*}
	F_{\mu} (w_n) \geq \beta_0||w_n||_{\dot{H}^{1/2}}^2 +o_n(1).
\end{equation*}
Therefore, we conclude that $||w_n||_{\dot{H}^{1/2}}^2 \to 0$, and $u_a\in\mathcal{M}_a$.
\end{proof}

%\section{Proof of Theorem \ref{th1.4}}
%In this section,
In the following, we derive a better upper bound of $m(a)$. We consider the problem
\begin{equation}\label{d1}
\sqrt{-\Delta}u =\bar{\lambda} u +\mu|u|^{q-2}u \ \ \text{in}\ \R^N
\end{equation}
under the constraint $S(a)$, where $\mu>0$ and $2<q<2+\frac{2}{N}$. Define $I(u)=\frac{1}{2}\|u\|^2_{\dot{H}^{1/2}}-\frac{\mu}{q}\|u\|^{q}_{q}$, then solutions $u$ of \eqref{d1} can be found as minimizers of
\begin{equation*}
m_0(a)=\inf_{u\in S(a)} I(u)>-\infty,
\end{equation*}
where $\bar{\lambda}$ is Lagrange multipliers. As in the proof of Theorem 1.2 of \cite{ZL} (see also \cite{FL}), we obtain that $m_0(a)$ has a unique positive radial minimizer $\alpha Q(\beta x)$ with
\begin{equation}\label{a22}
\alpha=\frac{a\beta^{\frac{N}{2}}}{\|Q\|_2}, \quad \beta=\Big[ \frac{\mu a^{q-2}}{\|Q\|^{q-2}_2} \Big]^{\frac{2}{2-q\gamma_q}},
\end{equation}
where $Q$ is the unique positive radial ground state solution of
\begin{equation}\label{a23}
 \sqrt{-\Delta}Q + Q = |Q|^{q-2}Q \ \ \text{in}\ \R^N.
\end{equation}

By a direct calculation, we have the following corollary.
\begin{Cor} \label{lem5.1}
Let $a>0$, $\mu>0$ and  $q\in(2,2+\frac{2}{N})$. Then \eqref{d1} has a unique positive solution $(\bar{\lambda}_0, u_0)$
given by
\begin{equation*}
\bar{\lambda}_0=-\|Q\|^{\frac{2(2-q)}{2-q\gamma_q}}_2\mu^{\frac{2}{2-q\gamma_q}}a^{\frac{2(q-2)}{2-q\gamma_q}},\quad
u_0:= \alpha Q(\beta x ),
\end{equation*}
where $\alpha,\beta$ are given in \eqref{a22}.
%$\alpha=\frac{a^{\frac{1}{2}}m^{\frac{N}{2}}}{\|Q\|_2}$ and $\beta=\Big[ \frac{\|Q\|^{q-1}_2}{\mu a^{\frac{q-1}{2}}} \Big]^{\frac{2}{N(q-1)-2}}$.
Furthermore, we have
$$m_0(a)=I(u_0)=-K_{N,q} a^{\frac{2q(1-\gamma_q)}{2-q\gamma_q}},$$
where $K_{N,q}:=\frac{(2-q\gamma_q)\gamma_q}{2q(1-\gamma_q)}\cdot \|Q\|^{\frac{2(2-q)}{2-q\gamma_q}}_2\mu^{\frac{2}{2-q\gamma_q}}>0$.
\end{Cor}
%\begin{proof}
%The proof is similar to that of Lemma 4.1 \cite{LY} and is omitted.
%\end{proof}

\begin{Lem} \label{lem5.2}
Let $a>0$, $\mu>0$ and  $q\in(2,2+\frac{2}{N})$ and $a\in(0, a_*]$. Then
\begin{equation*}
m(a)<-K_{N,q} a^{\frac{2q(1-\gamma_q)}{2-q\gamma_q}},
\end{equation*}
where $K_{N,q}=\frac{(2-q\gamma_q)\gamma_q}{2q(1-\gamma_q)}\cdot \|Q\|^{\frac{2(2-q)}{2-q\gamma_q}}_2\mu^{\frac{2}{2-q\gamma_q}}>0$.
\end{Lem}

\begin{proof}
Since
\begin{equation*}
I(u)=\frac{1}{2}\|u\|^2_{\dot{H}^{1/2}}-\frac{\mu}{q}\|u\|^{q}_{q}\ge g_a(\|u\|_{\dot{H}^{1/2}}),
\end{equation*}
where $g_a(\rho)=\frac{1}{2}\rho^2-\frac{\mu}{q}C^{q}_{opt} \rho^{q\gamma_q}a^{q(1-\gamma_q)}$. Note that $g_a(\rho)<0$ if $\rho\in (0, \tilde{\rho}_a)$ and $g_a(\rho)>0$ if $\rho\in (\tilde{\rho}_a,+\infty)$, where
$\tilde{\rho}_a:=\Big[\frac{2\mu}{q}C^{q}_{opt} a^{q(1-\gamma_q)}\Big]^{\frac{1}{2-q\gamma_q}}.$
%If $I(u_0)=m_0(a)<0$, it follows that
By Lemma \ref{lem5.1}, we get
\begin{equation*}
\begin{aligned}
\|u_0\|^{2}_{\dot{H}^{1/2}}&=\mu\gamma_q\|u_0\|^q_q\le\mu\gamma_qC^{q}_{opt}a^{q(1-\gamma_q)}\|u_0\|^{q\gamma_q}_{\dot{H}^{1/2}}.
%\alpha^2 \beta^{1-N}\|Q\|^{2}_{\dot{H}^{1/2}}=\frac{\gamma_q}{1-\gamma_q}\|Q\|^{\frac{2(2-q)}{2-q\gamma_q}}_2\mu^{\frac{2}{2-q\gamma_q}} a^{\frac{2q(1-\gamma_q)}{2-q\gamma_q}}.
\end{aligned}
\end{equation*}
It follows from Lemma \ref{lem3.1} that
\begin{equation*}
%\|u_0\|_2=a,\ \
 \|u_0\|_{\dot{H}^{1/2}}<\tilde{\rho}_a<R_0<\rho_0.
\end{equation*}
That is $u_0\in V(a)$, and $\displaystyle m(a)=\inf_{V(a)} F_\mu(u)\le F_\mu(u_0)<I(u_0)=m_0(a)$.
\end{proof}

Next, we prove the asymptotic behaviour of ground states to \eqref{eq1.1}-\eqref{ad1}.
\begin{Lem} \label{lem5.21}
For any fixed $\mu >0$, letting $a\to 0$, then for any ground state $u_{a}\in \mathcal{M}_a$, we have
\begin{equation*}
\frac{m(a)}{a^{\frac{2q(1-\gamma_q)}{2-q\gamma_q}}}\to -K_{N,q},\ \
\frac{\lambda_{a}}{a^{\frac{2(q-2)}{2-q\gamma_q}}}\to
-\frac{2q(1-\gamma_q)}{2-q\gamma_q}K_{N,q},\ \
\frac{\|u_{a}\|^2_{\dot{H}^{1/2}}}{a^{\frac{2q(1-\gamma_q)}{2-q\gamma_q}}}\to \frac{2q\gamma_q}{2-q\gamma_q}K_{N,q}.
%\quad \frac{\|u_{a_k}\|^{2^*}_{2^*}}{A}\to 0,
\end{equation*}
%\begin{equation*}
%\frac{\|u_{a_k}\|^2_{\dot{H}^{1/2}}}{A}\to \frac{2N(q-1)}{2-N(q-1)}K_{N,q},\ \quad %\frac{\|u_{a_k}\|^{q+1}_{q+1}}{A}\to -\frac{2(q+1)}{\mu\big[N(q-1)-2\big]}K_{N,q},
%\end{equation*}
%where $L=a^{\frac{N+q+1-Nq}{2-N(q-1)}}$ and $K_{N,q}=\frac{N+2-Nq}{2(N+q+1-Nq)}\cdot \|Q\|^{\frac{2(q-1)}{N(q-1)-2}}_2\mu^{\frac{2}{2-N(q-1)}}>0$.
Moreover, we have
\begin{equation*}
\frac{1}{\alpha} u_{a}\big(\frac{x}{\beta}\big)\to Q \ \  \text{in}\ H^{1/2}(\R^N )\ \ \text{as}\ a\to 0^+,
\end{equation*}
where $Q$ is the unique positive radial ground state of \eqref{a23},
$\alpha=\frac{a\beta^{\frac{N}{2}}}{\|Q\|_2}$, $\beta=\Big[ \frac{\mu a^{q-2}}{\|Q\|^{q-2}_2} \Big]^{\frac{2}{2-q\gamma_q}}$,  $K_{N,q}\!=\frac{(2-q\gamma_q)\gamma_q}{2q(1-\gamma_q)}\cdot \|Q\|^{\frac{2(2-q)}{2-q\gamma_q}}_2\mu^{\frac{2}{2-q\gamma_q}}>\!0$ and $\gamma_q=\frac{N(q-2)}{q}$.

%$F_\mu$ restricted to $S(a)$ has a ground state. This ground state is a local minimizer of $F_\mu$ in the set $V(a)$.
\end{Lem}
\begin{proof}
For $\mu >0$ fixed, let $a_k\to 0^+$ as $k\to +\infty$ and $u_{a_k}\in V(a_k)$ be a minimizer of $m(a_k)$ for each $k\in\mathbb{N}$, where $V(a_k)=\{ u\in S(a_k) : \|u\|_{\dot{H}^{1/2}}< \rho_0 \}$.
By Lemma \ref{lem3.3}, we get that $u_{a_k}$ is a ground state of $F_\mu\big|_{S(a_k)}$.
By Lemma $\ref{lem2.9}$, we can suppose that $\{u_{a_k}\}$ are nonnegative and radially symmetric, i.e., $0\le u_{a_k}\in H^{1/2}_r(\R^N)$. Then the Lagrange multipliers rule implies the existence of some $\lambda_{a_k} \in \R$ such that
\begin{equation}\label{d2}
\int_{\R^N}\sqrt{-\Delta}u_{a_k} \phi=\lambda_{a_k}\int_{\R^N}u_{a_k}\phi+
\mu\int_{\R^N}|u_{a_k}|^{q-2}u_{a_k}\phi+\int_{\R^N}|u_{a_k}|^{2^*-2}
u_{a_k}\phi
\end{equation}
for each $\phi\in H^{1/2}(\R^N)$.

We claim that
\begin{equation}\label{Dd2.2}
-\frac{1-\gamma_q}{\gamma_q}\Big[\frac{\mu 2N(2^*-q\gamma_q)}{q2^*}C^{q}_{opt}\Big]^{\frac{2}{2-q\gamma_q}}
a^{\frac{2(q-2)}{2-q\gamma_q}}_k<\lambda_{a_k}<-2K_{N,q}
a^{\frac{2(q-2)}{2-q\gamma_q}}_k.
\end{equation}
%\Big[\frac{\mu N(N+q+1-Nq)}{q+1}C^{q+1}_q \Big]^{\frac{2}{2-N(q-1)}}a^{\frac{q-1}{2-N(q-1)}}_k
In fact, it follows from \eqref{d2} that
\begin{equation*}
\begin{aligned}
\lambda_{a_k} a^2_k=\|u_{a_k}\|^2_{\dot{H}^{1/2}}-\mu\|u_{a_k}\|^{q}_{q}-\|u_{a_k}\|^{2^*}_{2^*}< -2K_{N,q}a^{\frac{2q(1-\gamma_q)}{2-q\gamma_q}}_k.
\end{aligned}
\end{equation*}
Since $P_\mu(u_{a_k})=0$ and Lemma \ref{lem5.2}, we have
\begin{equation*}
\begin{aligned}
F_\mu(u_{a_k})&=\frac{1}{2N}\|u_{a_k}\|^2_{\dot{H}^{1/2}}-\mu\frac{2^*-q\gamma_q}{q2^*}\|u_{a_k}\|^{q}_{q}\\
&=-\frac{2-q\gamma_q}{2q\gamma_q}\|u_{a_k}\|^2_{\dot{H}^{1/2}}+\frac{2^*-q\gamma_q}{2^*q\gamma_q}\|u_{a_k}\|^{2^*}_{2^*}\\
&< -K_{N,q}a^{\frac{2q(1-\gamma_q)}{2-q\gamma_q}}_k.
\end{aligned}
\end{equation*}
It follows immediately that
\begin{equation}\label{ba1}
\frac{2q\gamma_q}{2-q\gamma_q}K_{N,q}a^{\frac{2q(1-\gamma_q)}{2-q\gamma_q}}_k< \|u_{a_k}\|^2_{\dot{H}^{1/2}}<
\Big[\frac{\mu2N(2^*-q\gamma_q)}{q2^*}C^{q}_{opt}\Big]^{\frac{2}{2-q\gamma_q}}
a^{\frac{2q(1-\gamma_q)}{2-q\gamma_q}}_k.
\end{equation}
Hence, combining with $P_{\mu}(u_{a_k})=0$, we obtain that
\begin{equation*}
\begin{aligned}
\lambda_{a_k} a^2_k &=-\frac{1-\gamma_q}{\gamma_q}\|u_{a_k}\|^2_{\dot{H}^{1/2}}+\frac{1-\gamma_q}{\gamma_q}\|u_{a_k}\|^{2^*}_{2^*}\\
&>-\frac{1-\gamma_q}{\gamma_q}\Big[\frac{\mu 2N(2^*-q\gamma_q)}{q2^*}C^{q}_{opt}\Big]^{\frac{2}{2-q\gamma_q}}
a^{\frac{2q(1-\gamma_q)}{2-q\gamma_q}}_k.
\end{aligned}
\end{equation*}
We end the proof of claim \eqref{Dd2.2}.

Define $L:=a^{\frac{2q(1-\gamma_q)}{2-q\gamma_q}}_k$, we get
\begin{equation*}
\frac{2q\gamma_q}{2-q\gamma_q}K_{N,q}< \frac{\|u_{a_k}\|^2_{\dot{H}^{1/2}}}{L}< \Big[\frac{\mu 2N(2^*-q\gamma_q)}{q2^*}C^{q}_{opt}\Big]^{\frac{2}{2-q\gamma_q}}.
\end{equation*}
Then, we have
\begin{equation}\label{d3}
\begin{aligned}
\frac{\|u_{a_k}\|^{2^*}_{2^*}}{L}&\le \frac{\mathcal{S}^{\frac{2^*}{2}}\|u_{a_k}\|^{2^*}_{\dot{H}^{1/2}}}{L}\le\mathcal{S}^{\frac{2^*}{2}}\Big[\frac{\mu 2N(2^*-q\gamma_q)}{q2^*}C^{q}_{opt}\Big]^{\frac{2^*}{2-q\gamma_q}}L^{\frac{2^*}{2}-1}\to 0
\end{aligned}
\end{equation}
as $k\to \infty$. These facts imply that
\begin{equation}\label{d4}
\begin{aligned}
-K_{N,q}L&> m(a_k)=F_{\mu}(u_{a_k})=\frac{1}{2}\|u_{a_k}\|^2_{\dot{H}^{1/2}}-\frac{\mu}{q}\|u_{a_k}\|^{q}_{q}-\frac{1}{2^*}\|u_{a_k}\|^{2^*}_{2^*}\\
&\ge \inf_{u\in S(a_k)}\Big\{\frac{1}{2}\|u\|^2_{\dot{H}^{1/2}}-\frac{\mu}{q}\|u\|^{q}_{q}\Big\}-\frac{1}{2^*}\|u_{a_k}\|^{2^*}_{2^*}\\
&=-K_{N,q}L-\frac{1}{2^*}\|u_{a_k}\|^{2^*}_{2^*},
\end{aligned}
\end{equation}
where we use Lemma \ref{lem5.1} in the last equality. From \eqref{d3}-\eqref{d4}, we get
\begin{equation*}
\frac{m(a_k)}{L}\to -K_{N,q},\ \quad \frac{1}{2}\frac{\|u_{a_k}\|^2_{\dot{H}^{1/2}}}{L}-\frac{\mu}{q}\frac{\|u_{a_k}\|^{q}_{q}}{L}\to -K_{N,q}.
\end{equation*}
Since $P(u_{a_k})=0$, we get
\begin{equation*}
\frac{\|u_{a_k}\|^2_{\dot{H}^{1/2}}}{L}\to \frac{2q\gamma_q}{2-q\gamma_q}K_{N,q},\ \quad \frac{\|u_{a_k}\|^{q}_{q}}{L}\to \frac{2q}{\mu\big[2-q\gamma_q\big]}K_{N,q}.
\end{equation*}
Furthermore, we have
\begin{equation*}
\begin{aligned}
\frac{\lambda_{a_k}}{a^{\frac{2(q-2)}{2-q\gamma_q}}_k}
=\frac{\lambda_{a_k}a^2_k}{L}&=\frac{1}{L}\Big[\|u_{a_k}\|^2_{\dot{H}^{1/2}}
-\mu\|u_{a_k}\|^{q}_{q}-\|u_{a_k}\|^{2^*}_{2^*}        \Big]
\to-\frac{2q(1-\gamma_q)}{2-q\gamma_q}K_{N,q}<0.
\end{aligned}
\end{equation*}

Next, we give a precise description of $u_{a_k}$ as $k\to +\infty$. Define $v_{a_k}(x)=\frac{1}{\alpha_k} u_{a_k}(\frac{1}{\beta_k}x)$, where
\begin{equation*}
\alpha_k=\frac{a_k\beta^{\frac{N}{2}}_k}{\|Q\|_2},\quad
\beta_k=\Big[ \frac{\mu a^{q-2}_k}{\|Q\|^{q-2}_2}\Big]^{\frac{2}{2-q\gamma_q}}.
\end{equation*}
We then compute that
\begin{equation*}
\|v_{a_k}\|^2_{\dot{H}^{1/2}}=\frac{\beta^{N-1}_k}{\alpha^2_k}\|u_{a_k}\|^2_{\dot{H}^{1/2}}=\mu^{\frac{2}{q\gamma_q-2}}
\|Q\|^{\frac{2q(1-\gamma_q)}{2-q\gamma_q}}_2\frac{\|u_{a_k}\|^2_{\dot{H}^{1/2}}}{L},
\end{equation*}
\begin{equation*}
\|v_{a_k}\|^2_{2}=\frac{\beta^{N}_k}{\alpha^2_k}\|u_{a_k}\|^2_2=\frac{\|Q\|^2_2}{a^2_k}\|u_{a_k}\|^2_2=\|Q\|^2_2,
\end{equation*}
\begin{equation*}
\|v_{a_k}\|^{q}_{q}=\frac{\beta^{N}_k}{\alpha^{q}_k}\|u_{a_k}\|^{q}_{q}=\|Q\|^{\frac{2q(1-\gamma_q)}{2-q\gamma_q}}_2\mu^{\frac{q\gamma_q}{q\gamma_q-2}}\frac{\|u_{a_k}\|^{q}_{q}}{L}.
\end{equation*}
Therefore, $v_{a_k}$ is bounded in $H^{1/2}(\R^N)$. There exists $v\in H^{1/2}$ such that
\begin{equation*}
 v_{a_k}:=v_{a_k}(x)\rightharpoonup v \ \text{in}\ H^{1/2}.
\end{equation*}
We see that $v_{a_k}$ solves
\begin{equation}\label{d51}
\sqrt{-\Delta}v_{a_k}-\frac{\lambda_{a_k}}{\beta_k}v_{a_k}-\mu\frac{\alpha^{q-2}_k}{\beta_k}|v_{a_k}|^{q-2}v_{a_k}-\frac{\alpha^{2^*-2}_k}{\beta_k}|v_{a_k}|^{2^*-2}v_{a_k}=0.
\end{equation}
By a direct calculation, we deduce that
\begin{equation*}
-\frac{\lambda_{a_k}}{\beta_k}\to 1, \quad \mu\frac{\alpha^{q-2}_k}{\beta_k}\to 1, \quad \frac{\alpha^{2^*-2}_k}{\beta_k}=\mu^{\frac{2^*-2}{2-q\gamma_q}}\Big[\frac{a_k}{\|Q\|_2}\Big]^{\frac{(2^*-2)q(1-\gamma_q)}{2-q\gamma_q}}
\to 0.
\end{equation*}
Therefore, $v$ solves
\begin{equation}\label{d5}
\sqrt{-\Delta}v+v-|v|^{q-2}v=0.
\end{equation}
%Since $v\ge 0$ and $v\not\equiv0$, by the strong maximum principle, we get $v>0$.

We claim that the $v$ is a ground state of \eqref{d5}.

In fact, from Appendix of \cite{CW}, we get
$\|v\|^2_{\dot{H}^{1/2}}=\gamma_q\|v\|^{q}_{q}=\frac{\gamma_q}{1-\gamma_q}\|v\|^2_2$.
Thus, by Lemma \ref{lem2.2},
\begin{equation*}
\begin{aligned}
J(v)&=\Big[\frac{\|v\|^{\gamma_q}_{\dot{H}^{1/2}}\|v\|^
{1-\gamma_q}_2}{\|v\|_{q}}\Big]^{\frac{2q}{q-2}}=\Big[\frac{\gamma_q}{1-\gamma_q}\Big]^{\frac{q\gamma_q}{q-2}}(1-\gamma_q)^{\frac{2}{q-2}}\|v\|^2_2\\
%\Big[\frac{N(q-1)}{2-(N-1)(q-1)}\Big]^N  \Big[\frac{2-(N-1)(q-1)}{q+1}\Big]^{\frac{2}{q-1}}\|v\|^2_2\\
&\ge\frac{1}{C^{\frac{2q}{q-2}}_{opt}(N,q)}=\Big[\frac{\gamma_q}{1-\gamma_q}\Big]^{\frac{q\gamma_q}{q-2}}(1-\gamma_q)^{\frac{2}{q-2}}\|Q\|^2_2.
\end{aligned}
\end{equation*}
This implies $\|v\|^2_2=\|Q\|^2_2$ because $\displaystyle\|v\|^2_2\le\liminf_{k\to +\infty}\|v_{a_k}\|^2_2=\|Q\|^2_2$. Then, for any nontrivial solution $w$ of \eqref{d5}, we can deduce that
\begin{equation*}
\begin{aligned}
E(w)&=\frac{1}{2}\|w\|^2_{\dot{H}^{1/2}}+
\frac{1}{2}\|w\|^2_2-\frac{1}{q}\|w\|^{q}_{q}
=\frac{q-2}{2q(1-\gamma_q)}\|w\|^2_2\\
&\ge\frac{q-2}{2q(1-\gamma_q)}\|Q\|^2_2
=\frac{q-2}{2q(1-\gamma_q)}\|v\|^2_2=E(v).
\end{aligned}
\end{equation*}
This means that $v$ is a nonnegative radial ground state of \eqref{d5}. We can check that $v>0$. Otherwise, there exists $x_0\in \R^N$ such that $v(x_0)=0$. Then it follows from equation \eqref{d5} that $0=\sqrt{-\Delta} v(x_0)=C_{N} \mathrm{P.V.} \int_{\mathbb{R}^{N}} \frac{-v(y)}{|x-y|^{N+1}}dy$, which implies that $v\equiv0$ in $\R^N$. This is impossible since $\|v\|^2_2=\|Q\|^2_2$. So it must be $v=Q$ (up to a translation). Hence, by Lemma \ref{lem2.2} and $v_{a_k}\to Q$ in $L^2(\R^N)$, we have $v_{a_k}\to Q$ in $L^{q}(\R^N)$, and then $v_{a_k} \to Q$ in $\dot{H}^{1/2}(\R^N)$.
\end{proof}

\noindent \textbf{Proof of Theorem 1.1}
It follows from Lemma \ref{lem3.5} that there is a minimizer of $F_{\mu}$ on $V(a)$.
%The existence of a minimizer for $F_{\mu}$ on $V(a)$ follows from Lemma \ref{lem3.5}.
By Lemma \ref{lem3.3}, we see that the minimizer $u_a$ is indeed a ground state, and any ground state for $F_{\mu}$ on $S(a)$ belongs to $V(a)$. From \eqref{ba1}, we have $\|u_a\|_{\dot{H}^{1/2}}\to 0$ as $\mu \to 0$, then $F_{\mu}(u_a)\to 0$ as $\mu \to 0$ also holds. The item (2) of Theorem \ref{th1.1} follows from Lemma \ref{lem5.21}.
\qed
%\noindent \textbf{Proof of Theorem 1.1}
%It follows from Lemma \ref{lem3.5} that there is a minimizer of $F_{\mu}$ on $V(a)$.
%The existence of a minimizer for $F_{\mu}$ on $V(a)$ follows from Lemma \ref{lem3.5}.
%By Lemma \ref{lem3.3}, we see that the minimizer is indeed a ground state, and any ground state for $F_{\mu}$ on $S(a)$ belongs to $V(a)$.
%The fact that this minimizer is a ground state and that any ground state for $F_{\mu}$ on $S(a)$ belongs to $V(a)$ was proved in Lemma \ref{lem3.3}.
%\qed

\section{Proof of Theorem \ref{th1.2}}
In this section, we follow the approach of \cite{JL}. By Theorem 1.1, for any $\mu>0$ and $a\in (0, a_*(\mu)]$, $u_a$ is a minimizer of $F_\mu(u)|_{V(a)}$. Noting that $F_{\mu}$ is even and combining with Lemma \ref{lem2.9}, we can suppose that $u_a$ is a nonnegative and radially symmetric decreasing function. Similar to the proof of Lemma 2.8 of \cite{LY}, we get $u_a\in C^2(\R^N)$. Then by the strong maximum principle(\cite{SL}), we get $u_a>0$.
%There is a $\lambda_a\in \R$ such that
%\begin{equation*}
%\sqrt{-\Delta}u_a-\mu |u_a|^{q-1}u_a-u_a^3=\lambda_au_a \ \ \text{in}\ \R^N.
%\end{equation*}
%By Lemma \ref{lem2.7}, $P_\mu(u_a)=0$, it must hold that
%\begin{equation*}
%\lambda_a\|u_a\|^2_2=-\mu\frac{3-q}{q+1}\|u_a\|^{q+1}_{q+1},
%\end{equation*}
%which implies $\lambda_a<0$.
%From Lemma \ref{lem3.4} (2), we have the following corollaries.
\begin{Cor}\label{cor4.1}
For any $a \in (0, a_*]$, there exists a $d=d(a)>0$ such that $m(\alpha) \leq m(a) + d\ (a^2- \alpha^2)$ for any $\alpha \in \displaystyle [\frac{a}{2},a]$.
\end{Cor}
\begin{proof}
We can adopt a similar argument as the proof of Lemma 2.3 in \cite{JL}, thus we omit it.
%The proof is similar to that of Lemma \ref{lem3.4} and is omitted.
%We can adopt a similar argument as the proof in \cite{JJ} Lemma 2.3, thus we omit it.
%Let $u \in V(a)$ be a minimizer of $m(a)$ and set $y_{\alpha} := \sqrt{\frac{c- \alpha}{c}}\cdot u$. Since $y_{\alpha} \in V(c- \alpha)$ we have
%	\begin{align*}
%	m(c- \alpha) \leq F_{\mu}(y_{\alpha}) = m(c) + [ F_{\mu}(y_{\alpha}) - F_{\mu}(u)]
%	\end{align*}
%	with
%\begin{equation*}
%	F_{\mu}(y_{\alpha}) - F_{\mu}(u) = - \frac{\alpha}{2c}\|u\|_{\dot{H}^{1/2}}^2 - \frac{\mu}{q+1} [(\frac{a- \alpha}{a})^\frac{q+1}{2}-1 ]  \|u\|_{q+1}^{q+1} -  \frac{1}{4}[(\frac{a- \alpha}{a})^2 - 1 ] \|u\|^4_4.
%\end{equation*}
\end{proof}

As in Lemma \ref{lem3.1}, $t\star u(x)=t^{\frac{N}{2}}u(tx)$, $\Psi'_u(t)=\frac{1}{2t}P_\mu(t\star u)$ and $\Psi_u(t)$ has exactly two critical points $t^+_u$ and $t^-_u$. It is easy to show that $t^+_u$ is a local minimum point for $\Psi_u$ and $t^-_u$ is a global maximum point for $\Psi_u$. Furthermore, $F_\mu({t^+_u}\star u)<0< \inf_{u\in \partial V(a)}F_\mu(u)\le F_\mu({t^-_u}\star u)$.
\begin{Cor}\label{cor4.2}
$\Psi''_u(t^-_u)<0$ and the map $u\in S(a)\mapsto t^-_u\in \R $ is of class $C^1$.
\end{Cor}
\begin{proof}
By Lemma \ref{lem3.1}, we show that $\Psi_u''(t_u^-) <0$. $\Psi_u''(t)$ has a zero $t_u^0 \in (t_u^+, t_u^-)$. If not, $\Psi_u''(t)$ does not change the sign, this contradicts with $\Psi_u'(t_u^+)=\Psi_u'(t_u^-)=0$. %$\Psi_u(t_u^+)<0$ and $\Psi_u(t_u^-)>0$.
With a direct calculation,
$$ \Psi_u''(t) = -\frac{\mu\gamma_q}{2}\big[\frac{q\gamma_q}{2}-1\big]t^{\frac{q\gamma_q}{2}-2} \|u\|_{q}^{q} - \frac{1}{2(N-1)}t^{\frac{2^*}{2}-2} \|u\|_{2^*}^{2^*}.$$
Since $\frac{q\gamma_q}{2}<1$ and $\frac{2^*}{2}>1$, $\Psi_u''(t)$ has at most one zero and we are done.
Applying the Implicit Function Theorem to the $C^1$ function $g : \R \times S(a) \mapsto \R$ defined by $g(t,u) = \Psi_u'(t)$. Therefore, we have that
$u \mapsto t^-_u$ is of class $C^1$ because $g(t_u^-, u) =0$ and $\partial_t g(t_u^-, u) = \Psi_u''(t_u^-)<0$.
\end{proof}

We shall see that critical point of $\Psi_{u}(t)$ allow to project a function on $\Lambda(a)=\{u\in S(a): P_\mu(u)=0\}$. Thus, monotonicity and convexity properties of $\Psi_{u}(t)$ strongly affects the structure of $\Lambda(a)$. According to the above Corollary \ref{cor4.2}, we consider the decomposition of $\Lambda(a)$ into the disjoint union $\Lambda(a)=\Lambda^+(a)\cup\Lambda^-(a)$, where
\begin{equation*}
\Lambda^+(a):= \{u\in\Lambda(a): F_\mu(u)<0\}=\{u\in S(a):\Psi'_u(1)=0, F_\mu(u)<0 \},
\end{equation*}
\begin{equation*}
\Lambda^-(a):= \{u\in\Lambda(a): F_\mu(u)>0\}=\{u\in S(a):\Psi'_u(1)=0, F_\mu(u)>0 \}.
\end{equation*}

\begin{Lem}\label{lem4.2}
For any $a \in (0, a_*]$, we denote by
\begin{align*}
W(a) := \{u \in S(a): t^-_u > 1\}.
\end{align*}
The following properties hold.
\begin{enumerate}
		\item  $\Lambda^+(a) \subset W(a)$.
		\item  $\partial W(a) =\Lambda^-(a)$ and $\displaystyle \inf_{u \in \partial W(a)} F_{\mu}(u) >0.$
		\item  $V(a) \cap \{u : F_{\mu}(u) <0 \} \subset W(a).$
		\item  $\displaystyle \inf_{u \in W(a)} F_{\mu}(u)$ is reached and $\displaystyle \inf_{u \in W(a)} F_{\mu}(u)
                 =\inf_{u \in \Lambda^+(a)} F_{\mu}(u)= m(a)$.
\end{enumerate}
\end{Lem}
\begin{proof}
We can adopt a similar argument as the proof of Lemma 2.5 in \cite{JL}, thus we omit it.
\end{proof}

We denote $H^{1/2}_r(\R^N)$ by
$H^{1/2}_r(\R^N):=\big\{u\in H^{1/2}(\R^N): u(x)=u(|x|), x\in \R^N\big\},$
and by $S_r(a) := S(a) \cap H^{1/2}_r(\R^N)$.
Let
\begin{equation*}
M^0(a):=\inf_{h\in \Gamma^0(a)}\max\limits_{t\in[0,\infty)}F_\mu(h(t)),
\end{equation*}
where
\begin{equation*}
\Gamma^0(a):=\big\{h\in C([0,\infty),S_r(a)): h(0)\in \Lambda^+(a), \exists t_h \ \text{such that} \ h(t)\in E_a, \forall t\ge t_h \big\}
\end{equation*}
with
\begin{equation*}
E_a:=\big\{u\in S(a): F_\mu(u)<2m(a)\big\}\neq\emptyset.
\end{equation*}

%Now, we prove the existence of second solutions of \eqref{eq1.1}-\eqref{eq1.2} via the min-max.

It is not difficult to recognize that $F_{\mu}$ has a mountain pass geometry at level $M^0(a)$, and thus we can search a mountain pass critical point. This relies heavily on a refined version of the min-max principle introduced by N. Ghoussoub \cite{GN}, the forth coming Lemma \ref{lem4.3}, and was already applied in \cite{NsAe,NSaE}.

\begin{Def}\label{def3.9}
Let $B$ be a closed subset of $X$. We shall say that a class $\mathcal{F}$ of compact subsets of $X$ is a homotopy-stable family with extended boundary $B$ if for any set $A$ in $\mathcal{F}$ and any $\eta\in C([0,1]\times X;X)$ satisfying $\eta(t,x)=x$ for all $(t,x)\in (\{0\}\times X)\cup ([0,1]\times B)$ we have that $\eta(\{1\}\times A)\in \mathcal{F}$.
\end{Def}

\begin{Lem}(\cite{GN}, Theorem 5.2)\label{lem4.3}
Let $\varphi$ be a $C^{1}$-functional on a complete connected $C^{1}$-Finsler manifold $X$ and consider a homotopy-stable family $\mathcal{F}$ with an extended closed boundary $B$. Set $\displaystyle c=c(\varphi,\mathcal{F})=\inf_{A \in \mathcal{F}}\max\limits_{x\in A}\varphi(x)$ and let $F$ be a closed subset of $X$ satisfying \\
\indent $(1)$~~~~~~~~$(A \cap F)\backslash B \neq \emptyset \ \ \text { for every } A \in \mathcal{F}$, \\
\indent $(2)$~~~~~~~~$\sup \varphi(B) \leq c \leq \inf \varphi(F)$.  \\
Then, for any sequence of sets $(A_{n})_{n}$ in $\mathcal{F}$ such that $\lim_{n}\sup_{A_{n}}\varphi=c$, there exists a sequence $(x_{n})_{n}$ in $X\setminus B $ such that
$$\lim_{n \rightarrow +\infty}\varphi(x_{n})=c,\ \ \lim_{n \rightarrow +\infty}\|d\varphi(x_{n})\|=0,\ \ \lim_{n \rightarrow +\infty}dist(x_{n},F)=0,\ \ \lim_{n \rightarrow +\infty}dist(x_{n},A_{n})=0.$$
\end{Lem}

We introduce the functional $\tilde{F}_\mu: \R^+\times H^{1/2}(\R^N)\to \R$
\begin{equation*}
\tilde{F}_\mu(t,u):=F_\mu(t\star u)=\Psi_u(t)=\frac{t}{2}\|u\|^2_{\dot{H}^{1/2}}-\mu\frac{ t^{\frac{q\gamma_q}{2}}}{q}\|u\|^{q}_{q}-\frac{t^{\frac{2^*}{2}}}{2^*}\|u\|^{2^*}_{2^*}.
\end{equation*}
For fixed $a > 0$, we show that $\tilde{F}_\mu$ has a mountain pass geometry on $\R^+ \times S_r(a)$ at level $\tilde{M}(a)$, i.e.
\begin{align*}
	\tilde{M}(a) := \inf_{\tilde{h} \in \tilde{\Gamma}(a)} \max\limits_{t \in [0,\infty)} \tilde{F}_\mu(\tilde{h}(t)) > \max\limits \{\tilde{F}_\mu(\tilde{h}(0)), \tilde{F}_\mu(\tilde{h}(t_{\tilde{h}}))\}
\end{align*}
where
\begin{align*}
	\tilde{\Gamma}(a) :&=\big \{\tilde{h} \in \mathit{C}([0,\infty), \R^+ \times S_r(a)) : \tilde{h}(0) \in (1, \Lambda^+(a)), \exists t_{\tilde{h}} > 0 \\
 &\qquad\qquad \text{ such that} \ \tilde{h}(t) \in (1, E_a) \, \forall t \geq t_{\tilde{h}}\big\}.
\end{align*}

\begin{Lem}\label{lem4.4}
For any $a\in (0,a_*]$, $\tilde{F}_\mu$ has a mountain pass geometry at the level $\tilde{M}(a)$. Moreover, $M^0(a)=\tilde{M}(a)$.
\end{Lem}
\begin{proof}
The proof is similar to that of Lemma 3.3 in \cite{JL} and hence omit the details.

\end{proof}

\begin{Prop} \label{prop1}
	For any $a \in (0,a_*]$, there exists a Palais-Smale sequence $\{u_n\} \subset S_r(a)$ for $F_{\mu}$ restricted to $S_r(a)$ at level $M^0(a)$, with $P_\mu(u_n) \to 0$ as $n \to \infty$.
\end{Prop}
\begin{proof}
We verify the conditions of Lemma \ref{lem4.3} one by one. Let
\begin{align*}
 \mathcal{F}& = \big\{\tilde{h}([0,\infty)): \tilde{h} \in \tilde{\Gamma}(a) \big\}, \qquad B = (1, \Lambda^+(a)) \cup (1, E_a),\\
 F &= \{(t,u) \in \R^+ \times S_r(a): \tilde{F}_{\mu}(t, u) \geq \tilde{M}(a)\}.
\end{align*}
 We need to check that $\mathcal{F}$ is a homotopy stable family of compact subsets of $X$ with extended closed boundary $B$, and that $F$ is a dual set for $\mathcal{F}$, in the sense that assumptions (1) and (2) in Lemma \ref{lem4.3} are satisfied. By Lemma \ref{lem4.4}, we know $F\cap B=\emptyset$ and hence $F\setminus B=F$,   	
\begin{align*}
\sup_{(1, u) \in B} \tilde{F}_\mu(t, u) \leq \tilde{M}(a) \leq \inf_{(t, u) \in F} \tilde{F}_\mu(t, u).
\end{align*}
For any $A\in \mathcal{F}$, there exists $h_0\in  \tilde{\Gamma}(a)$ such that $A=h_0([0,\infty))$ and
\begin{align*}
	\tilde{M}(a) = \inf_{\tilde{h} \in \tilde{\Gamma}(a)} \max\limits_{t \in [0,\infty)} \tilde{F}_\mu(\tilde{h}(t)) \leq \max\limits_{t \in [0,\infty)} \tilde{F}_\mu(h_0(t)),
\end{align*}
we have
\begin{equation*}
A\cap F\setminus B=A\cap F\neq \emptyset, \quad \forall A\in \mathcal{F}.
\end{equation*}
Now, for all $(t, u) \in \R^+ \times S_r(a)$, we have
\begin{align*}
\tilde{F}_\mu(t, u) = F_{\mu}(t\star u) = \tilde{F}_\mu(1, t\star u).
\end{align*}
Hence, for any minimizing sequence $(z_n = (\alpha_n, \beta_n) ) \subset \tilde{\Gamma}(a)$ for $\tilde{M}(a)$, the sequence $(y_n = (1, \alpha_n\star\beta_n) )$ is also a minimizing sequence for $\tilde{M}(a)$. Consequently, by Lemma \ref{lem4.3}, there exists a Palais-Smale sequence $\{(t_n,w_n)\}\subset \mathbb{R}^+\times S_r(a)$ for $\tilde{F}_{\mu}|_{\mathbb{R}^+\times {S_r(a)}}$ at level $\tilde{M}(a)>0$ such that
\begin{equation}  \label{eqq7.7}
\partial_{t} \tilde{F_{\mu}}\left(t_{n}, w_{n}\right) \rightarrow 0 \quad \text { and } \quad\big\|\partial_{u} \tilde{F_{\mu}}\left(t_{n}, w_{n}\right)\big\|_{\left(T_{w_{n}} S_{r}(a)\right)^{'}} \rightarrow 0 \quad \text { as } n \rightarrow \infty,
\end{equation}
with the additional property that
\begin{equation}   \label{eqq7.8}
\left|t_{n}-1\right|+\operatorname{dist}_{H^{1/2}}\big(w_{n}, \alpha_n\star\beta_{n}([0,\infty))\big) \rightarrow 0 \quad \text { as } n \rightarrow \infty.
\end{equation}
The first condition in (\ref{eqq7.7}) reads $P_{\mu}\left( t_n\star w_{n}\right) \rightarrow 0$, and the second condition in (\ref{eqq7.7}) gives that $\forall \varphi \in T_{w_{n}} S_{r}(a)$
\begin{equation*} %\label{eqq7.9}
\begin{aligned}
 t_n \int_{\mathbb{R}^{N}} \sqrt{\!-\!\Delta} w_{n} \varphi- \mu t_n^{\frac{q\gamma_q}{2}}\int_{\mathbb{R}^{N}} \left|w_{n}\right|^{q\!-\!2} w_{n} {\varphi}-t^{\frac{2^*}{2}}_n \int_{{\mathbb{R}^N}} |w_n|^{2^*-2}w_n\varphi =\!o(1)|| \varphi||_{H^{1/2}}.
\end{aligned}
\end{equation*}
Since $\{t_n\}$ is bounded from above and from below,
due to (\ref{eqq7.8}), we have
\begin{equation} \label{equa4.7}
d F_{\mu}\left({t_n}\star w_n\right)\left[t_n\star\varphi\right]=o(1)\|\varphi\|_{H^{1/2}}=o(1)\left\| t_n\star\varphi\right\|_{H^{1/2}}~~\text{as}~~n \rightarrow \infty, \ \ \forall \varphi \in T_{w_{n}} S_{r}(a).
\end{equation}
By Lemma \ref{lem2.8}, (\ref{equa4.7}) implies that $\{u_n:=t_n\star w_n\} \subset S_{r}(a)$ is a Palais-Smale sequence for $F_{\mu}|_{S_{r}(a)}$ (thus a Palais-Smale sequence for $F_{\mu}|_{S(a)}$, since the problem is invariant under rotations) at level $M^0(a)>0$, with $P_{\mu}(u_n)\to 0$.	
\end{proof}

\begin{Prop} \label{prop2}
For any $a \in (0,a_*]$, if
\begin{align}\label{b1}
M^0(a) < m(a) + \frac{\mathcal{S}^N}{2N},
\end{align}
then the Palais-Smale sequence obtained in Proposition \ref{prop1} is, up to subsequence, strongly convergent in $H^{1/2}_r(\R^N)$.	
\end{Prop}
\begin{proof}
The proof is divided into four main steps.

\noindent (1) Boundedness of $\{u_n\}$ in $H^{1/2}_r(\R^N)$.
\noindent Since $P_{\mu}(u_n)=||u_n||_{\dot{H}^{1/2}}^2
-\mu\gamma_q{||u_n||}_{q}^{q}-\|u_n\|^{2^*}_{2^*}=o_n(1)$. By Lemma \ref{lem2.2}, we have
\begin{align*}
F_{\mu}\left(u_{n}\right)&=\frac{1}{2}||u_{n}||_{\dot{H}^{1/2}}^{2}-\frac{\mu}{q}\|u_n\|^{q}_{q}-\frac{1}{2^*}\|u_n\|^{2^*}_{2^*}\\
&\ge\frac{1}{2N}||u_{n}||_{\dot{H}^{1/2}}^{2}-\frac{\mu(2^*-q\gamma_q)}{q2^*}C^{q}_{opt} a^{q(1-\gamma_q)}||u_{n}||_{\dot{H}^{1/2}}^{q\gamma_q}
+o_n(1).
\end{align*}
Since $2<q<2+\frac{2}{N}$, we get $ F_{\mu}\left(u_{n}\right)\le M^0(a)+1$ as $n$ large enough.
This implies that $||u_{n}||_{\dot{H}^{1/2}}\leq C$. So $\{u_n\}$ is bounded in $H^{1/2}$ because $||u_{n}||_{2}=a$.

\noindent (2) $\exists$ Lagrange multipliers $\lambda_{n} \rightarrow \lambda \in \mathbb{R}$. %Since the embedding $H_{r}^{1/2}\left(\mathbb{R}^{N}\right) \hookrightarrow L^{r}\left(\mathbb{R}^{N}\right)$ is compact for $r\in\big(2,\frac{2N}{N-1}\big)$,
By Lemma \ref{lem2.1}, we deduce that there exists $u \in H_{r}^{1/2}(\R^N)$ such that, up to a subsequence, $u_{n} \rightharpoonup u$ weakly in $H^{1/2}_r(\R^N)$, $u_{n} \rightarrow u$ strongly in $L^{r}\left(\mathbb{R}^{N}\right)$ for $r \in\big(2,\frac{2N}{N-1}\big)$, and a.e. in $\R^N$. Now, since $\left\{u_{n}\right\}$ is a Palais-Smale sequence of $\left.F_{\mu}\right|_{S(a)}$, by the Lagrange multipliers rule there exists $\lambda_{n} \in \mathbb{R}$ such that
\begin{equation} \label{eq4.2}
\int_{\mathbb{R}^{N}} [\sqrt{-\Delta}u_{n} \cdot {\varphi}-\lambda_{n} u_{n}{\varphi}]-\mu\int_{\mathbb{R}^{N}} \left|u_{n}\right|^{q\!-\!2} u_{n} {\varphi}-\int_{{\mathbb{R}^N}}|u_n|^{2^*-2}u_n\varphi\!=\!o_n(1)(|| \varphi||_{H^{1/2}})
\end{equation}
for every $\varphi \in H^{1/2}(\R^N)$, where $o_n(1) \rightarrow 0$ as $n \rightarrow \infty$. In particular, take $\varphi=u_n$, then
$$\lambda_{n} a^2=||u_n||_{\dot{H}^{1/2}}^2
-\mu{||u_n||}_{q}^{q}-\|u_n\|^{2^*}_{2^*}+o_n(1),$$
and the boundedness of $\left\{u_{n}\right\}$ in $H^{1/2}_r(\R^N)$ implies that $\left\{\lambda_{n}\right\}$ is bounded as well; thus, up to a subsequence $\lambda_{n} \rightarrow \lambda \in \mathbb{R} .$

\noindent (3) We claim that $\lambda<0$ and $u\not \equiv0$. Recalling that $P_{\mu}\left(u_{n}\right) \rightarrow 0$, we have
$$\lambda_{n} a^2=-\mu(1-\gamma_q)||u_{n}||_{q}^{q}+o(1).$$
Let $n\to+\infty$, then $\lambda a^2=-\mu(1-\gamma_q)||u||_{q}^{q}$. Since $\mu>0$, we deduce that $\lambda \leq 0$, with equality if and only if $u\equiv0$. If $\lambda_{n} \rightarrow 0$, we have $\mathop {\lim }\limits_{n  \to \infty}||u_{n}||_{q}^{q}=0$. Using again $P_{\mu}\left(u_{n}\right) \rightarrow 0$, we have $\mathop {\lim }\limits_{n  \to \infty} ||u_n||_{\dot{H}^{1/2}}^2=\mathop {\lim }\limits_{n  \to \infty}{||u_n||}_{2^*}^{2^*}=\ell$. Therefore, by the Sobolev inequality $\ell \geq {\mathcal{S}}  \ell^{\frac{N-1}{N}}$.
We have $\ell=0$ or $\ell \geq {\mathcal{S}}^{N}$. Since
$$0\not =M^0(a)=\lim_{n\to+\infty} F_{\mu}\left(u_{n}\right)=\lim_{n\to+\infty}\Big[\frac{1}{2}
{||u_{n}||}_{\dot{H}^{1/2}}^2-\frac{\mu}{q}\|u_n\|^{q}_{q}
-\frac{1}{2^*} {||u_{n}||}_{2^*}^{2^*}\Big]=\frac{1}{2N}\ell,$$
we have $\ell \not =0$ and $\ell \geq {\mathcal{S}}^{N}$. Then we have $M^0(a)=\mathop {\lim }\limits_{n  \to \infty} F_{\mu}\left(u_{n}\right)=\frac{1}{2N}\ell \geq\frac{1}{2N} {\mathcal{S}}^{N}$,
and this contradicts our assumptions $M^0(a)<m(a)+\frac{1}{2N} {\mathcal{S}}^{N}<\frac{1}{2N} {\mathcal{S}}^{N}$. Therefore, we have $\lambda <0$ and $u\not\equiv0$.

\noindent (4) $u_{n} \rightarrow u$ in $H^{1/2}_r(\R^N)$. Since $u_{n} \rightharpoonup u\not\equiv0$ weakly in $H^{1/2}_r(\R^N)$, then (\ref{eq4.2}) imply that
\begin{equation} \label{equa3.7}
  d F_{\mu}(u) \varphi-\lambda \int_{\mathbb{R}^{N}} u {\varphi}=0,~~~~~~~~\forall \varphi \in H^{1/2}(\R^N).
\end{equation}
That is $u$ is a weak radial (and real) solution to
\begin{equation}\label{ac1}
\sqrt{-\Delta} u=\lambda u+\mu|u|^{q-2} u+|u|^{2^*-2} u ~~~~\ \ \mbox{in}~~~~{\R}^N.
\end{equation}
Therefore, we have $P_{\mu}(u)=0$.
%||u||_{\dot{H}^{1/2}}^2-\mu\gamma_q||u||_{q}^{q}-\|u\|^{2^*}_{2^*}=0$.
Denote $v_{n}=u_{n}-u$, then $ v_{n} \rightharpoonup 0$ in $H^{1/2}_r\left(\mathbb{R}^{N}\right)$ and therefore
$$||u_n||_{\dot{H}^{1/2}}^2=||u||_{\dot{H}^{1/2}}^2+||v_n||_{\dot{H}^{1/2}}^2+o(1).$$
By the Br\'ezis-Lieb lemma in \cite{WM}, we have
$$F_\mu(u_n)=F_\mu(u)+F_\mu(v_n)+o_n(1),
~~~~~~~~P_\mu(u_n)=P_\mu(u)+P_\mu(v_n)+o_n(1).$$
Since $v_{n} \rightarrow 0$ strongly in $L^{q}(\R^N)$, $P_{\mu}\left(u_{n}\right)=o(1)$ and $P_{\mu}(u)=0$, we deduce that $\mathop  \|v_n\|_{\dot{H}^{1/2}}^2
=\mathop \|v_n\|_{2^*}^{2^*}+\lambda\|v_n\|^2_2+o_n(1)$.
Therefore, by the Sobolev inequality $S\|v_n\|^2_{2^*}\le \|v_n\|_{\dot{H}^{1/2}}^2\le \|v_n\|_{2^*}^{2^*}+o_n(1)$.
We have $\|v_n\|^{2^*}_{2^*}=0$ or $\|v_n\|^{2^*}_{2^*} \geq {\mathcal{S}}^N$.

If $\|v_n\|^{2^*}_{2^*}\geq {\mathcal{S}}^N$, since $a\mapsto m(a)$ is non-increasing, we have
\begin{equation*}
\begin{aligned}
M^0(a)=\lim_{n\to+\infty} F_{\mu}\left(u_{n}\right)&=F_{\mu}\left(u\right)+ \frac{1}{2N}\|v_n\|_{\dot{H}^{1/2}}^2 +o_n(1)\\ % \geq F_{\mu}\left(u\right)+\frac{1}{2N} {\mathcal{S}}^{N}\\
&\ge m(\|u\|_2)+\frac{1}{2N} {\mathcal{S}}^{N}\\
&\ge m(a)+\frac{1}{2N} {\mathcal{S}}^{N},
\end{aligned}
\end{equation*}
which is a contradiction.

If instead $\|v_n\|^{2^*}_{2^*} =0$, then $ u_{n} \rightarrow u$ in $\dot{H}^{1/2}_r(\mathbb{R}^{N})$. In order to prove that $ u_{n} \rightarrow u$ in $L^2\left(\mathbb{R}^{N}\right)$, let $\varphi=u_{n}-u$ in (\ref{equa3.7}), test (\ref{ac1}) with $u_n-u$, and then subtract, we get
\begin{equation*}
\begin{aligned}
||u_{n}-u||_{\dot{H}^{1/2}}^{2}-\int_{\mathbb{R}^{N}}(\lambda_{n} u_{n}-\lambda u)(u_{n}-u)&-\mu \int_{{\mathbb{R}^N}}\big(|u_n|^{q-2}u_n-|u|^{q-2}u\big)(u_n-u) \\
&=\int_{\mathbb{R}^{N}}
\left(\left|u_{n}\right|^{2^*-2} u_{n}-|u|^{2^*-2} u\right)\left(u_{n}-u\right)+o(1).
\end{aligned}
\end{equation*}
Then, we have
%$$\int_{{\mathbb{R}^N}} \left( I_{\alpha} \ast|u_n|^{p}\right)|u_n|^{p\!-\!2}u_n(u_n-u)\leq ||u_{n}||_{\frac{2Np}{N+\alpha}}^{2p-1}||u_{n}-u||_{\frac{2Np}{N+\alpha}}\to 0 $$
%and
%$$\int_{{\mathbb{R}^N}} \left( I_{\alpha} \ast|u|^{p}\right)|u|^{p\!-\!2}u(u_n-u)\leq ||u||_{\frac{2Np}{N+\alpha}}^{2p-1}||u_{n}-u||_{\frac{2Np}{N+\alpha}}\to 0. $$
%Consequently,
$$
0=\lim _{n \rightarrow \infty} \int_{\mathbb{R}^{N}}\left(\lambda_{n} u_{n}-\lambda u\right)\left(u_{n}-u\right)=\lim _{n \rightarrow \infty} \lambda \int_{\mathbb{R}^{N}}\left(u_{n}-u\right)^{2},
$$
and we deduce that $ u_{n} \rightarrow u$ strongly in $H^{1/2}_r(\mathbb{R}^{N})$.
Therefore, the proposition is proved.

\end{proof}

\begin{Prop} \label{prop3}
Let $ 2<q<2+\frac{2}{N}$. For any $a \in (0,a_*]$, it holds that
\begin{align*}
M^0(a)<m(a) + \frac{\mathcal{S}^N}{2N}.
\end{align*}
\end{Prop}

Proposition \ref{prop3} is vital in proving Theorem \ref{th1.2} and its proof can be divided into two parts via different type of arguments.
Let
\begin{align*}
M(a) := \inf_{\gamma \in \Gamma(a)} \max\limits_{t \in [0,\infty)} F_{\mu}(\gamma(t)),
\end{align*}
where
\begin{align*}
\Gamma(a) := \big\{\gamma \in \mathit{C}([0,\infty), S(a)) &: \gamma(0) \in V(a) \cap \{ u: F_{\mu}(u) <0\}, \\
&\quad \exists t_\gamma \ \text{such that}\  \gamma(t) \in E_a, \, \forall t \geq t_\gamma\big\}.
\end{align*}
\begin{Prop} \label{prop4}
For any $a\in (0,a_*]$,  it holds that
\begin{align*}
M^0(a) = M(a).
\end{align*}
\end{Prop}
\begin{proof}
Since $\Gamma^0(a)\subset \Gamma(a)$, we have $M(a)\le M^0(a)$.
%As in the proof of Proposition 1.15 \cite{JJ} gives that $M(a)\ge M^0(a)$.
We proceed to show that $M(a)\ge M^0(a)$. %and the proof consists of three steps.
\medbreak
\noindent (1) For any $a \in (0, a_*]$, we have
\begin{align*}
M(a) \geq \inf_{u \in \Lambda^-(a)} F_{\mu}(u).
\end{align*}

Let $h \in \Gamma(a)$, $h(0) \in V(a) \cap \{u : F_{\mu}(u) <0 \}$, by Lemma \ref{lem4.2}, then we have $h(0) \in W(a)$ and $t_{h(0)}^- > 1$.  Since $h \in \Gamma(a)$, we also have that $F_{\mu}(h(t)) \leq 2m(a) < m(a)$ for $t$ large enough. Thus, using Lemma \ref{lem4.2} again, we get that $h(t) \not \in W(a)$ for $t$ large enough or equivalently that $t_{h(t)}^-<1$ for such $t >0$. By Corollary \ref{cor4.1}, the continuity of $h$ and of $u \mapsto t_u^-$, we deduce that there exists a $t_0 >0$ such that $t_{h(t_0)}^- = 1$, ie., $h(t_0) \in \partial W(a)$. From Lemma \ref{lem4.2}, we get
\begin{align*}
M(a) \geq \inf_{u \in \partial W(a)} F_{\mu}(u) = \inf_{u \in \Lambda^-(a)} F_{\mu}(u).
\end{align*}

\medbreak
\noindent (2) For any $a \in (0, a_*]$, it holds that
\begin{align*}
\inf_{u \in \Lambda^-(a)} F_{\mu}(u) \geq \inf_{u \in \Lambda^{-}_r(a)} F_{\mu}(u).
\end{align*}

For any $u \in \Lambda^-(a)$, let $u^*$ be the Schwartz rearrangement of $|u|$. By Lemma \ref{lem2.9}, we get that $\Psi_{u^*}(t) \leq \Psi_u(t)$ for all $t \geq 0$. Indeed,
since $t^{-}_u$ is the unique global maximum point for $\Psi_u$, we have
\begin{align*}
\Psi_u(t^{-}_u) \geq \Psi_u(t_{u^*}^-) \geq  \Psi_{u^*}(t_{u^*}^-).
\end{align*}
Since $u \in \Lambda^-(a)$, we have $t^{-}_{u}= 1$ and hence
\begin{align*}
F_{\mu}(u) = \Psi_u(1) = \Psi_u(t^{-}_u) \geq \Psi_{u^*}(t_{u^*}^-) = F_{\mu}({t_{u^*}^-}\star u^*).
\end{align*}
It follows from ${t_{u^*}^-}\star u^* \in\Lambda^{-}_r(a)$ that
\begin{align*}
\inf_{u \in \Lambda^-(a)} F_{\mu}(u) \geq \inf_{u \in \Lambda^-_r(a)} F_{\mu}(u).
\end{align*}

\medbreak
\noindent (3) For any $a \in (0, a_*]$, it results that
\begin{align*}
\inf_{u \in \Lambda^{-}_r(a)} F_{\mu}(u) \geq M^0(a).
\end{align*}

Let $u \in \Lambda^{-}_r(a)$ and $t_1 > 0$ be such that ${t_1}\star u \in E_a$. We consider the map
\begin{align*}
g_u : s \in [0, \infty) \mapsto ((1-s)t^+_u + st_1)\star u \in S_r(a).
\end{align*}
Then $g_u \in \mathit{C}([0,\infty), S_r(a))$ and
\begin{align*}
g_u(0) = {t^+_u}\star u \in \Lambda^+(a) \quad\mbox{and}\quad g_u(1) ={t_1}\star u \in E_a.
\end{align*}
Therefore, we get $g_u \in \Gamma^0(a)$ and
\begin{align*}
F_{\mu}(u) = \max\limits_{t >0} F_{\mu} (t\star u) \geq \max\limits_{s \in [0,\infty)} F_{\mu}(g_u(s)) \geq \inf_{g \in \Gamma^0(a)}  \max\limits_{s \in [0,\infty)} F_{\mu}(g(s)) = M^0(a).
\end{align*}
\end{proof}

\begin{Prop} \label{prop5}
Let $2<q<2+\frac{2}{N}$. For any $a \in (0,a_*]$, we have
\begin{align*}
M(a) < m(a) + \frac{\mathcal{S}^N}{2N}.
\end{align*}
\end{Prop}

We introduce some necessary Lemmas.
\begin{Lem} \label{lem4.5}
For any $a \in (0,a_*]$, it holds that
$$M(a) \leq \inf_{h \in \mathcal{G}(a)} \max\limits_{t \in [0,1]} F_{\mu}(h(t)),$$
where
\begin{align*}
\mathcal{G}(a) := \Big\{h \in \mathit{C}([0,\infty), \displaystyle \cup_{d \in [\frac{a}{2},a]} S(d)) &: h(0) \in \mathcal{M}_d \, \text{ for some } \, d \in [\frac{a}{2},a], \\
&\exists t_0 = t_0(h) \ \mbox{ such that}\ h(t) \in E_a, \, \forall t \geq t_0 \Big\}.
\end{align*}
\end{Lem}

\begin{proof}
For any $h \in \mathcal{G}(a)$. Setting
\begin{align*}
t \mapsto \theta(t) :=\frac{\|h(t)\|^2_{L^2}}{a^2}.
\end{align*}
Then we easily see that $\theta$ is the continuous function and $\theta(t) \leq 1$ for all $t$. Let
\begin{align*}
g(t)(x) := \theta(t)^{\frac{N-1}{2}} h(t)(\theta(t)x).
\end{align*}
It holds that
\begin{align*}
\|g(t)\|_{2}^2 &= \frac{1}{\theta(t)} \|h(t)\|_{2}^2 = a^2, &\|g(t)\|_{\dot{H}^{1/2}}^2 &= \|h(t)\|_{\dot{H}^{1/2}}^2, \\
\|g(t)\|_{q}^{q} &= [\theta(t)]^{(\frac{(N-1)q}{2}-N)} \|h(t)\|_{q}^{q}, &\|g(t)\|_{2^*}^{2^*} &= \|h(t)\|_{2^*}^{2^*}.
\end{align*}
Therefore,
\begin{align*}
F_{\mu} (g(t)) &= \frac{1}{2} \|g(t)\|_{\dot{H}^{1/2}}^2 - \frac{\mu}{q} \|g(t)\|_{q}^{q} - \frac{1}{2^*}\|g(t)\|_{2^*}^{2^*} \\
&= \frac{1}{2} \|h(t)\|_{\dot{H}^{1/2}}^2 - \frac{\mu}{q} [\theta(t)]^{(\frac{(N-1)q}{2}-N)} \|h(t)\|_{q}^{q} - \frac{1}{2^*}\|h(t)\|_{2^*}^{2^*}\\
&\leq \frac{1}{2} \|h(t)\|_{\dot{H}^{1/2}}^2 - \frac{\mu}{q} \|h(t)\|_{q}^{q}- \frac{1}{2^*}\|h(t)\|_{2^*}^{2^*} = F_{\mu} (h(t)),
\end{align*}
due to $\theta(t) \leq 1$ for all $t$ and $\frac{(N-1)q}{2}-N < 0$.
Since $F_{\mu}(g(0)) \leq F_{\mu}(h(0)) <0$ and $\|g(0)\|_{\dot{H}^{1/2}} = \|h(0)\|_{\dot{H}^{1/2}} < \rho_0$, we have $g(0) \in V(a) \cap \{u : F_{\mu}(u) < 0\}$ and hence that  $g \in \Gamma(a)$. This implies the Proposition.
\end{proof}

By Lemma \ref{lem2.1}, we know that $\mathcal{S}=\displaystyle\inf_{u\in\dot{H}^{1/2}}\frac{\|u\|^2_{\dot{H}^{1/2}}}{\|u\|^{2}_{2^*}}$ is attained by
%Let $u_{\varepsilon}$ be an extremal function for the Sobolev inequality in $\R^N$ defined by
\begin{align*}
u_{\varepsilon}(x) :=C(N) \frac{\varepsilon^{\frac{N-1}{2}}}{(\varepsilon^2+|x|^2)^{\frac{N-1}{2}}}, \quad \varepsilon > 0, \quad x\in\R^N,
\end{align*}
where $C(N)$ is a positive constant. Take a radially decreasing cut-off function $\xi \in C_0^{\infty}(\R^N)$ such that $\xi \equiv 1$ in $B_1$, $\xi \equiv 0$ in $\R^N \backslash B_2$, and let $U_{\varepsilon}(x) = \xi(x) u_{\varepsilon}(x)$. %We shall prove the following lemma.
Similar to the proof of Propositions 14.11--14.12 of \cite{GvRR}, we can deduce that
\begin{equation*}
||U_{\varepsilon}||_{\dot{H}^{1/2}}^2 \!\leq \! {\mathcal{S}}
^{N}\!+\!O(\varepsilon^{N-1}),\qquad ||U_{\varepsilon}||_{2^*}^{2^*}\!=\!{\mathcal{S}}
^N\!+\!O(\varepsilon^{N}),~~~~
\end{equation*}
and
\begin{equation*}
||U_{\varepsilon}||_{2}^2=\left\{\begin{array}{ll}{O(\varepsilon)} & {\text { if } N >2}, \\
 {O(\varepsilon|\log{\varepsilon}|) } & {\text { if } N\!=2}.\end{array}\right.~~~~
\end{equation*}

Let $u \in H^{1/2}(\R^N)$ be a nonnegative function. For any
$\varepsilon >0$  and  $ t > 0$, we have
\begin{equation}\label{e11}
\begin{aligned}
F_{\mu} (u + t U_{\varepsilon} ) & \leq  F_{\mu}(u)+t \int_{\R^N}\int_{\R^N}\frac{(u(x)-u(y))(U_\varepsilon(x)-U_\varepsilon(y))}{|x-y|^{N+1}}
	+\frac{t^2}{2} \|U_{\varepsilon}\|_{\dot{H}^{1/2}}^2 \\
&\qquad- \frac{\mu t^{q}}{q} \|U_{\varepsilon}\|_{q}^{q}- \frac{t^{2^*}}{2^*} \|U_{\varepsilon}\|_{2^*}^{2^*}.
\end{aligned}
\end{equation}
In the following, we fix a sequence  $\{\varepsilon_n\} \subset \R^+ $ such that $\varepsilon_n \to 0$.
Set
\begin{equation*}
\begin{aligned}
I_n(t)&:= t \int_{\R^N}\int_{\R^N}\frac{(u_n(x)-u_n(y))(U_{\varepsilon_n}(x)-U_{\varepsilon_n}(y))}{|x-y|^{N+1}}
+\frac{t^2}{2} \|U_{\varepsilon_n}\|_{\dot{H}^{1/2}}^2-\frac{\mu t^{q}}{q} \|U_{\varepsilon_n}\|_{q}^{q}
- \frac{t^{2^*}}{2^*} \|U_{\varepsilon_n}\|_{2^*}^{2^*},
\end{aligned}
\end{equation*}
where $\{u_n\}\subset H^{1/2}(\R^N)$ satisfies
\begin{equation*}
\int_{\R^N}\int_{\R^N}\frac{(u_n(x)-u_n(y))(U_{\varepsilon_n}(x)-U_{\varepsilon_n}(y))}{|x-y|^{N+1}}  \leq 1, \quad \forall n \in N.
\end{equation*}
Then we assert that there exists $t_1> 0$ such that, for any $n \in N$ large enough,
\begin{equation}\label{e1}
I_n(t) \leq 2 m(a)\ \ \text{for\ any}\ t \geq t_1.
\end{equation}

\begin{Lem}\label{lem4.8}
Let $a \in (0,a_*]$ and $u_a\in \mathcal{M}_a=\big\{u\in V(a) \big| F_{\mu}(u)=m(a)\big\}$. For any $\varepsilon >0$, there exists a $y_{\varepsilon} \in \R^N$ such that
\begin{align}\label{b2}
2 \int_{\R^N}u_a( x - y_{\varepsilon}) U_{\varepsilon}(x) \leq t_1  \|U_{\varepsilon}\|_{2}^2,
\end{align}
where $t_1$ is given by \eqref{e1} and
\begin{align}\label{b3}
\int_{\R^N}\int_{\R^N}\frac{(u(x-y_\varepsilon)-u(y-y_\varepsilon))(U_\varepsilon(x)-U_\varepsilon(y))}{|x-y|^{N+1}}\leq  \|U_{\varepsilon}\|_{2}^2.
\end{align}
\end{Lem}
\begin{proof}
Let $u_a \in \mathcal{M}_a$ be a radial, non-increasing function, by Radial Lemma A.IV of \cite{BL}, we have
\begin{equation}\label{b4}
|u_a(z)| \leq C|z|^{-\frac{N}{2}} \|u_a\|_2=Ca|z|^{-\frac{N}{2}}, \quad \forall |z| \geq 1.
\end{equation}
%A direct calculation shows that
%\begin{align}\label{b5}
%\int_{\R^N}u_a( x - y) U_{\varepsilon}(x)&\leq C a \int_{\R^N}  \frac{U_{\varepsilon}(x)}{|x-y|^{N/2}}.
%\end{align}
Then for $|y|$ large enough, we get
\begin{equation}\label{abb1}
\begin{aligned}
\int_{\R^N}u_a( x - y) U_{\varepsilon}(x)&\leq C a \int_{\R^N}  \frac{U_{\varepsilon}(x)}{|x-y|^{N/2}}\\
&%\le Ca\int_{\R^N}  \frac{U_{\varepsilon}(x)}{|x-y|^{N/2}}
\le Ca \int_{B_2} \frac{U_{\varepsilon}(x)} {\big||y|-R \big|^{N/2}} \leq Ca \big|\frac{y}{2}\big|^{-\frac{N}{2}}\int_{B_2}  U_{\varepsilon}(x).
\end{aligned}
\end{equation}
Here, we use the fact that $U_{\varepsilon}$ has compactly supported in $B_2$. Thus, it follows from \eqref{abb1} that
\begin{align*}
\int_{\R^N}u_a(x-y)U_{\varepsilon}(x) \leq C \sqrt{a} \big|\frac{2}{y}\big|^{\frac{2}{N}} \int_{B_2}U_{\varepsilon}(x)\le C\sqrt{a} \big|\frac{2}{y}\big|^{\frac{2}{N}} |B_2| \|U_{\varepsilon}\|^2_2.
\end{align*}
%Using H\"{o}lder's inequality, we obtain that $\|U_{\varepsilon}\|_{1} \leq |B_2|^{\frac{1}{2}} \|U_{\varepsilon}\|_2$ and \eqref{b2} follows.
Then letting $|y|$ be large enough such that $C\sqrt{a} \big|\frac{2}{y}\big|^{\frac{2}{N}} |B_2| <t_1$, we prove that \eqref{b2} holds. For any $y \in \R^N$, $u_a(\cdot-y) \geq0 $ is solution to the equation
\begin{align*}
\sqrt{-\Delta} u - \mu |u|^{q-2} u - |u|^{2^*-2} u = \lambda_a u \quad \text{in}\ \R^N,
\end{align*}
for some $\lambda_a<0$, we have that
\begin{align*}
\begin{split}
\int_{\R^N}\int_{\R^N} \frac{(u_a(x -y)-u_a(z-y))(U_{\varepsilon}(x)-U_{\varepsilon}(z))}{|x-z|^{N+1}}  &\leq
 \mu \int_{\R^N} |u_a(x-y)|^{q-1} U_{\varepsilon}(x) \\
&\qquad+ \int_{\R^N}|u_a(x -y)|^{2^*-1}U_{\varepsilon}(x).
\end{split}
\end{align*}
From \eqref{b4}, we see that $|u_a(z)|^{q-1} \leq |u_a(z)|$ and $|u_a(z)|^{2^*-1} \leq |u_a(z)|$ for $|z|$ large. Similar arguments apply to the case \eqref{b3}.	
\end{proof}

The sequence $\{a_n\} \in \displaystyle \big[\frac{a}{2}, a\big)$ is defined as
\begin{equation}\label{definitioncn}
a^2_n := a^2- 2 t_1^2 \|U_{\varepsilon_n}\|_{2}^2,
\end{equation}
where $t_1>0$ is given in \eqref{e1}. Clearly, $a_n \to a$ as $n \to \infty$. For each $n \in N $, we fix a $u_{a_n} \in \mathcal{M}_{a_n}$.

%\begin{Lem}\label{lem4.9}
%Under the setting introduced above, for any $n\in N$ large enough, there exists a $y_n \in \R^N$ such that
%\begin{equation}\label{c1}
%a_n \leq \|u_{a_n}(\cdot - y_{n}) + t U_{\varepsilon_n}\|_{2}^2 \leq a, \quad \forall t\in [0,t_1],
%\end{equation}
%and
%\begin{equation}\label{c2}
%\int_{\R^N}\int_{\R^N}\frac{(u_{a_n}(x-y_{n})-u_{a_n}(z-y_{n})) (U_{\varepsilon_n}(x)-U_{\varepsilon_n}(z) )}{|x-y|^{N+1}} \leq \max\limits \{1, \|U_{\varepsilon_n}\|_{2}^2 \}.
%\end{equation}
%\end{Lem}

\begin{Lem}\label{lem4.10}
Under the setting introduced above,  we denote by
\begin{align*}
\gamma_n(t) =
\begin{cases}
\displaystyle u_{a_n}(\cdot - y_n) + t U_{\varepsilon_n} \quad \text{if} \quad t \in [0,t_1],\\
\displaystyle \gamma_n(t_1) \quad \text{if} \quad t \geq t_1,
\end{cases}
\end{align*}
where $|y_n|\to \infty$ as $n\to \infty$. Then we have $\gamma_n \in \mathcal{G}(a)$ and
\begin{equation}\label{key}
M(a) \leq \max\limits_{t \in [0, \infty)} F_{\mu}(\gamma_n(t)) \leq \max\limits \Big\{ \max\limits_{t \in [t_0,t_1]}F_{\mu}(\gamma_n(t)), \,
m(a) + \frac{1}{3N}\mathcal{S}^{N} \Big\},
\end{equation}
for any $n \in N$ large enough, where $t_0>0$.
\end{Lem}
\begin{proof}
Let $n \in N$ be arbitrary but fixed. %We first prove that $\gamma_n \in \mathcal{G}(a)$.
Clearly, $\gamma_n \in C([0, \infty), H^{1/2})$, by Lemma \ref{lem4.8}, we have
\begin{equation*}
\begin{aligned}
\|u_{a_n}(\cdot - y_{n}) + t U_{\varepsilon_n}\|_{2}^2&=\|u_{a_n}(\cdot - y_{n})\|^2_2+2t\int_{\R^N}u_{a_n}(x-y_n)U_{\varepsilon_n}(x)+t^2\|U_{\varepsilon_n}\|^2_2\\
&\le a^2_n+2t_1 \int_{\R^N}u_{a_n}(x-y_n)U_{\varepsilon_n}(x)+t^2_1\|U_{\varepsilon_n}\|^2_2\\
&\le a^2_n+2t^2_1\|U_{\varepsilon_n}\|^2_2=a^2,
\end{aligned}
\end{equation*}
and hence
$$ \gamma_n(t) \subset \bigcup_{d \in [\frac{a}{2}, a]}S(d)$$
since $a_n \to a$ as $n\to \infty$. From \eqref{e11},
\begin{equation}\label{c3}
\begin{aligned}
F_{\mu} (u_{a_n} + t U_{\varepsilon_n} ) & \leq  m(a_n)+t \int_{\R^N}\int_{\R^N}\frac{(u_{a_n}(x)-u_{a_n}(y))(U_{\varepsilon_n}(x)-U_{\varepsilon_n}(y))}{|x-y|^{N+1}}\\
&\qquad	+\frac{t^2}{2} \|U_{\varepsilon_n}\|_{\dot{H}^{1/2}}^2
- \frac{\mu t^{q}}{q} \|U_{\varepsilon_n}\|_{q}^{q}- \frac{t^{2^*}}{2^*} \|U_{\varepsilon_n}\|_{2^*}^{2^*}.
\end{aligned}
\end{equation}
Since $\|U_{\varepsilon_n}\|^2_2\to 0$ as $n\to \infty$,
\begin{equation}\label{c2}
\int_{\R^N}\int_{\R^N}\frac{(u_{a_n}(x-y_{n})-u_{a_n}(z-y_{n})) (U_{\varepsilon_n}(x)-U_{\varepsilon_n}(z) )}{|x-y|^{N+1}} \leq \max\limits \{1, \|U_{\varepsilon_n}\|_{2}^2 \}.
\end{equation}
We can apply \eqref{e1} to deduce that, for $t \geq t_1$,
\begin{equation}\label{star2}
F_{\mu}(\gamma_n(t)) \leq m(a_n) + 2 m(a) \leq 2 m(a).
\end{equation}
We conclude that $\gamma_n \in \mathcal{G}(a)$ for any $n \in N$ large enough.

In particular the first inequality in \eqref{key} holds because of Lemma \ref{lem4.5}. Since $t\to 0$, $I_n(t)<\frac{1}{4N}\mathcal{S}^{N}$. There exists $t_0>0$ such that if $I_n(t) \geq \frac{1}{4N}\mathcal{S}^N$ then necessarily $t \geq t_0$. Now, considering again \eqref{c3} and recording that \eqref{c2} apply, if $t \in [0,t_0]$, we get
\begin{equation}\label{star3}
F_{\mu}(\gamma_n(t)) \leq m(a_n) + \frac{1}{4N} \mathcal{S}^{N} \leq m(a) + \frac {1}{3N} \mathcal{S}^{N}
\end{equation}
since $a_n \to a$ as $n \to \infty$. Combining \eqref{star2} and \eqref{star3} we see that the second inequality in \eqref{key} holds.

\end{proof}

\begin{proof}[Proof of Proposition \ref{prop5}]
We assume the setting above. In order to prove
$$M(a) < m(a) + \frac{1}{2N}\mathcal{S}^{N},$$
by Lemma \ref{lem4.10}, it suffices to show that
$$
\max\limits_{t \in [t_0,t_1]}F_{\mu} (u_{a_n}(\cdot- y_n) + t U_{\varepsilon_n}) < m(a) + \frac{1}{2N}\mathcal{S}^{N}.
$$
From Corollary \ref{cor4.1}, Lemma \ref{lem4.8} and the definition of $a_n$ given in \eqref{definitioncn}  we can write
\begin{align*}
\begin{split}
\max\limits_{t \in [t_0,t_1]}&F_{\mu} (u_{a_n}(\cdot -y_n) + t U_{\varepsilon_n} )  \\
\leq \, &m(a_n) + \max\limits_{t \in [t_0,t_1]} [t \|U_{\varepsilon_n}\|_{2}^2 + \frac{t^2}{2} \|U_{\varepsilon_n}\|_{\dot{H}^{1/2}}^2 - \frac{\mu t^{q}}{q} \|U_{\varepsilon_n}\|_{q}^{q}-\frac{t^{2^*}}{2^*} \| U_{\varepsilon_n}\|_{2^*}^{2^*}]\\
\leq \, &m(a) + d(a^2  - a^2_n) + t_1 \|U_{\varepsilon_n}\|_{2}^2
	+ \max\limits_{t \in [t_0,t_1]} [\frac{t^2}{2} \|U_{\varepsilon_n}\|_{\dot{H}^{1/2}}^2  - \frac{\mu t^{q}}{q} \| U_{\varepsilon}\|_{q}^{q}-\frac{t^{2^*}}{2^*} \| U_{\varepsilon_n}\|_{2^*}^{2^*} ] \\
\leq \, &m(a) +   2t_1^2 d  \|U_{\varepsilon_n}\|_{\dot{H}^{1/2}}^2   + t_1 \|U_{\varepsilon_n}\|_{2}^2
+ \max\limits_{t \in [t_0,t_1]} [\frac{t^2}{2} \|U_{\varepsilon_n}\|_{\dot{H}^{1/2}}^2  - \frac{\mu t^{q}}{q} \| U_{\varepsilon_n}\|_{q}^{q} - \frac{t^{2^*}}{2^*} \| U_{\varepsilon_n}\|_{2^*}^{2^*}].
\end{split}
\end{align*}
It remains to prove that, for $n$ sufficiently large,
\begin{align*}
J_n:=  (2t_1^2 d + t_1)  \|U_{\varepsilon_n}\|_{2}^2 + \max\limits_{t \in [t_0,t_1]} [\frac{t^2}{2} \|U_{\varepsilon_n}\|_{\dot{H}^{1/2}}^2 - \frac{\mu t^{q}}{q} \|U_{\varepsilon_n}\|_{q}^{q} - \frac{t^{2^*}}{2^*} \|U_{\varepsilon_n}\|_{2^*}^{2^*}]
< \frac{1}{2N}\mathcal{S}^{N}.
\end{align*}
We proceed to show that
$$J_n \leq (2t_1^2 d + t_1)  \|U_{\varepsilon_n}\|_{2}^2 + \max\limits_{t >0} [\frac{t^2}{2}\|U_{\varepsilon_n}\|_{\dot{H}^{1/2}}^2 - \frac{t^{2^*}}{2^*} \|U_{\varepsilon_n}\|_{2^*}^{2^*}] -\frac{\mu t_0^{q}}{q} \|U_{\varepsilon_n}\|_{q}^{q}.$$
It is well know that	
\begin{equation*}\label{c4}
\max\limits_{t >0} \big[\dfrac{t^2}{2} \|U_{\varepsilon_n}\|_{\dot{H}^{1/2}}^2 - \frac{t^{2^*}}{2^*} \|U_{\varepsilon_n}\|_{2^*}^{2^*}\big] =
\frac{1}{2N}\mathcal{S}^{N} + O(\varepsilon_n).
\end{equation*}

To complete the proof of Proposition \ref{prop5}, we only need to prove that, for $n$ large enough,
%The proof of Proposition \ref{prop5} is completed by showing that,
\begin{equation*}\label{SecondX2}
(2t_1^2 d + t_1)  \|U_{\varepsilon_n}\|_{2}^2  - \frac{\mu t_0^{q}}{q} \|U_{\varepsilon_n}\|_{q}^{q}
+  O(\varepsilon_n)<0.
\end{equation*}
From Lemma 7.1 of \cite{JL}, we have
\begin{equation*}
%\|U_{\varepsilon_n}\|_2^2=C\varepsilon_n|\log\varepsilon_n|+O(\varepsilon_n),\quad
\|U_{\varepsilon_n}\|_{q}^{q}=C\varepsilon_n^{N-\frac{(N-1)q}{2}}+o\big(\varepsilon_n^{N-\frac{(N-1)q}{2}}\big).
\end{equation*}
If $N>2$, there exist constants $K_1 >0$ and $K_2>0$, for $n $ sufficiently large, we have
\begin{equation*}
\begin{aligned}
(2t_1^2 d + t_1) \|U_{\varepsilon_n}\|_{2}^2 - \frac{\mu t_0^{q}}{q} \|U_{\varepsilon_n}\|_{q}^{q}+ O(\varepsilon_n)&=K_1 \, \varepsilon_n -K_2\, \varepsilon_n^{N-\frac{(N-1)q}{2}}+ O(\varepsilon_n)\\
&\leq - \frac {K_2}{2} \, \varepsilon_n^{N-\frac{(N-1)q}{2}}<0.\\
\end{aligned}
\end{equation*}
Similarly, if $N=2$, there exist constants $\tilde{K}_1 >0$ and $\tilde{K}_2>0$ such that
\begin{equation*}
\begin{aligned}
(2t_1^2 d + t_1) \|U_{\varepsilon_n}\|_{2}^2 - \frac{\mu t_0^{q}}{q} \|U_{\varepsilon_n}\|_{q}^{q}+ O(\varepsilon_n)&=\tilde{K}_1 \, \varepsilon_n|\log \varepsilon_n|-\tilde{K}_2\, \varepsilon_n^{N-\frac{(N-1)q}{2}}+ O(\varepsilon_n)\\
&\leq - \frac {\tilde{K}_2}{2} \, \varepsilon_n^{N-\frac{(N-1)q}{2}}<0,
\end{aligned}
\end{equation*}
since $\frac{1}{N}<N-\frac{(N-1)q}{2}< 1 \Longleftrightarrow 2<q<2+\frac{2}{N}. $ \medskip
Therefore, we  deduce that the conclusion of Proposition \ref{prop5} holds.	
	
\end{proof}

%\section{Proof of Theorem \ref{th1.3}}

\noindent \textbf{Proof of Theorem 1.2.}\ \
%The proof of Theorem \ref{th1.2} (1) follows directly combining
By Propositions \ref{prop1}-\ref{prop3}, there exists a second solution $v_a\in V(a)$ which satisfies $F_{\mu}(v_a)<m(a)+\frac{\mathcal{S}^{N}}{2N}$.

%have a precise description of the asymptotic behavior of $v_a$ when $a\to 0$.
%\begin{proof}[Proof of Theorem \ref{th1.3}]
At first, for $\mu >0$ fixed, we prove that
\begin{equation}\label{c5}
F_{\mu}(v_a) \to \frac{\mathcal{S}^{N}}{2N} \quad \text{as} \quad  a \to 0.
\end{equation}
By $P_{\mu}(v_a) = 0$, as in Proposition \ref{prop2}
\begin{align*}
F_{\mu}\left(v_{a}\right)&=\frac{1}{2}||v_{a}||_{\dot{H}^{1/2}}^{2}-\frac{\mu}{q}\|v_a\|^{q}_{q}-\frac{1}{2^*}\|v_a\|^{2^*}_{2^*}\\
&\ge\frac{1}{2N}||v_{a}||_{\dot{H}^{1/2}}^{2}-\frac{\mu(2^*-q\gamma_q)}{q2^*}C^{q}_{opt} a^{q(1-\gamma_q)}||v_{a}||_{\dot{H}^{1/2}}^{q\gamma_q}
+o(1).
\end{align*}
Since
\begin{equation*}
F_{\mu}(v_a) = M^0(a) < m(a) + \frac{\mathcal{S}^{N}}{2N} \leq \frac{\mathcal{S}^{N}}{2N}
\end{equation*}
and $q\gamma_q<2$, we deduce that $\{v_a\} \subset H^{1/2}(\R^N)$ is uniformly bounded with respect to $a \in (0,a_*].$ Thus, by the fractional Gagliardo-Nirenberg inequality, we have
\begin{align*}
\|v_a\|_{q}^{q} \leq C^{q}_{opt}(N,q) \|v_a\|_{2}^{q(1-\gamma_q)} \|v_{a}\|_{\dot{H}^{1/2}}^{q\gamma_q} \to 0 \quad\mbox{as } a \to 0.
\end{align*}
Hence, from $P_{\mu}(v_a)=0$, we get
\begin{align*}
\ell := \lim_{a\to 0} ||v_{a}||_{\dot{H}^{1/2}}^{2}= \lim_{a\to 0} \|v_a\|_{2^*}^{2^*} \leq \frac{1}{\mathcal{S}^{\frac{2^*}{2}}} \lim_{a\to 0} ||v_{a}||_{\dot{H}^{1/2}}^{2^*} = \frac{1}{\mathcal{S}^{\frac{2^*}{2}}} \ell^{\frac{2^*}{2}}.
\end{align*}
Therefore, either $\ell = 0$ or $\ell \geq \mathcal{S}^{N}$. We claim that $\ell = 0$  is impossible. Indeed, since  $v_a \in \Lambda^-(a)$, we have that
\begin{align*}
||v_{a}||_{\dot{H}^{1/2}}^{2} - \mu \gamma_q \|v_a\|_{q}^{q} - \|v_a\|_{2^*}^{2^*} = 0,
\end{align*}
and by Corollary \ref{cor4.2},
\begin{align*}
-\frac{\mu\gamma_q}{2}\big[\frac{q\gamma_q}{2}-1\big]\|v_a\|_{q}^{q} - \frac{1}{2(N-1)}\|v_a\|_{2^*}^{2^*}< 0.
\end{align*}
Using the Sobolev inequality, we get
\begin{align*}
||v_{a}||_{\dot{H}^{1/2}}^{2} < K\|v_a\|_{2^*}^{2^*} \leq K \frac{1}{\mathcal{S}^{\frac{2^*}{2}}} \|v_a\|_{\dot{H}^{1/2}}^{2^*},
\end{align*}
where $K=1+\frac{2}{(N-1)(2-q\gamma_q)}$. Therefore, $\ell\ge \mathcal{S}^N$, and $P_{\mu}(v_a)=0$, we have
\begin{align*}
M^0(a) = F_{\mu} (v_a) &=  \frac{1}{2N} ||v_{a}||_{\dot{H}^{1/2}}^{2} - \frac{\mu(2^*-q\gamma_q)}{q2^*} \|v_a\|_{q}^{q}\\
&= \frac{1}{2N} ||v_{a}||_{\dot{H}^{1/2}}^{2} + o_a(1)
\geq \frac{\mathcal{S}^{N}}{2N}+o_a(1).
\end{align*}
Moreover $m(a) \to 0$ as $a \to 0$, and that
\begin{align*}
M^0(a) < m(a) + \frac{\mathcal{S}^{N}}{2N},
\end{align*}
we obtain \eqref{c5}, which implies that $||v_{a}||_{\dot{H}^{1/2}}^{2}\to \mathcal{S}^{N}$. We conclude the proof of Theorem \ref{th1.2} (1).
	
By Lemma \ref{lem3.1}, we have $a_*(\mu)\to \infty$ as $\mu\to 0$. Since $v_a\in S(a)$ exists for any $\mu\to 0$ sufficiently small. Similarly, we get $\{v_a\}\subset H^{1/2}_r(\R^N)$ is uniformly bounded as $\mu \to 0$ and thus, using the fractional Gagliardo-Nirenberg inequality, we have  	
\begin{align*}
\mu \|v_a\|_{q}^{q}\leq \mu C^q_{opt}(N,q) \|v_a\|_{2}^{q(1-\gamma_q)} ||v_{a}||_{\dot{H}^{1/2}}^{q\gamma_q} \to 0 \quad\text{as} \ \mu\to 0.
\end{align*}
The rest of the proof is similar to the one of Theorem \ref{th1.2} (1).
\qed
%\end{proof}

\end{document}